\documentclass[final]{siamart0216}
\usepackage{amsmath,amsfonts}
\usepackage{enumerate}
\usepackage{general}
\usepackage{hyperref}
\usepackage[bold]{hhtensor}
\usepackage[latin1]{inputenc}
\usepackage{nicefrac}
\usepackage{paralist}
\usepackage{parskip}
\usepackage[font=footnotesize]{subcaption}
\usepackage{tikz}
\usepackage{pgfplots,pgfplotstable}

\newcommand{\TheTitle}{A Hybrid High-Order method for the Cahn--Hilliard problem in mixed form}
\newcommand{\TheAuthors}{F. Chave, D. A. Di Pietro, F. Marche, F. Pigeonneau}

\ifpdf
\hypersetup{
  pdftitle={\TheTitle},
  pdfauthor={\TheAuthors}
}
\fi

\newcommand{\dt}{\delta_t}
\newcommand{\bdf}{\delta_t^{(2)}}

\newcommand{\UT}[1][k]{\underline{U}_T^{#1}}
\newcommand{\Uh}[1][k]{\underline{U}_h^{#1}}
\newcommand{\UhO}[1][k]{\underline{U}_{h,0}^{#1}}

\newcommand{\IT}[1][k]{\underline{I}_T^{#1}}
\newcommand{\Ih}[1][k]{\underline{I}_h^{#1}}

\newcommand{\uchn}[1][n]{\underline{c}_h^{#1}}
\newcommand{\uwhn}[1][n]{\underline{w}_h^{#1}}

\newcommand{\uhchn}[1][n]{\widehat{\underline{c}}_h^{#1}}
\newcommand{\uhwhn}[1][n]{\widehat{\underline{w}}_h^{#1}}

\newcommand{\hchn}[1][n]{\widehat{c}_h^{#1}}
\newcommand{\hwhn}[1][n]{\widehat{w}_h^{#1}}

\newcommand{\uech}[1][n]{\underline{e}_{c,h}^{#1}}
\newcommand{\uewh}[1][n]{\underline{e}_{w,h}^{#1}}

\newcommand{\ech}[1][n]{e_{c,h}^{#1}}
\newcommand{\ewh}[1][n]{e_{w,h}^{#1}}

\newcommand{\uphi}[1][h]{\underline{\varphi}_{#1}}
\newcommand{\upsi}[1][h]{\underline{\psi}_{#1}}

\newcommand{\uv}{\underline{v}_T}
\newcommand{\uvh}{\underline{v}_h}

\newcommand{\uz}{\underline{z}_T}
\newcommand{\uzh}{\underline{z}_h}
\newcommand{\uhzh}{\underline{\widehat{z}}_h}
\newcommand{\hzh}{\widehat{z}_h}

\newcommand{\pT}[1][k+1]{p_T^{#1}}

\newcommand{\Lh}[1][k]{L_h^{#1}}
\newcommand{\uLh}[1][k]{\underline{L}_h^{#1}}

\newcommand{\G}{\mathcal{G}}
\newcommand{\Gh}[1][k]{\mathcal{G}_h^{#1}}
\newcommand{\uGh}[1][k]{\underline{\mathcal{G}}_h^{#1}}

\newcommand{\tF}{t_{\rm F}}

\newcommand{\lproj}[2][h]{\pi_{#1}^{#2}}

\newcommand{\jump}[2][F]{[#2]_{#1}}
\newcommand{\avg}[2][F]{\{#2\}_{#1}}

\makeatletter
\def\thm@space@setup{%
  \thm@preskip=\parskip \thm@postskip=0pt
}
\makeatother

\headers{A HHO method for the Cahn--Hilliard problem}{\TheAuthors}

\title{{\TheTitle}\thanks{This work was partially supported by Saint-Gobain Recherche (contract UM 150095). D. Di Pietro also acknowledges the partial support of Agence Nationale de la Recherche project HHOMM (ANR-15-CE40-0005).}}

\author{
  Florent Chave\thanks{University of Montpellier, Institut Montpelli\'{e}rain Alexander Grothendieck, 34095 Montpellier, France (\email{florent.chave@outlook.fr}, \email{daniele.di-pietro@umontpellier.fr}, \email{fabien.marche@umontpellier.fr})}
  \and 
  Daniele A. Di Pietro\footnotemark[2]
  \and
  Fabien Marche\footnotemark[2]~\thanks{INRIA Lemon team, 860 rue Saint-Priest 34095 Montpellier, France}
  \and
  Franck Pigeonneau\thanks{Surface du Verre et Interfaces, UMR 125 CNRS/Saint-Gobain, 93303 Aubervilliers Cedex, France (\email{franck.pigeonneau@saint-gobain.com})}
}

\begin{document}

\maketitle

\begin{abstract}
  In this work we develop a fully implicit Hybrid High-Order algorithm for the Cahn--Hilliard problem in mixed form.
  The space discretization hinges on local reconstruction operators from hybrid polynomial unknowns at elements and faces.
  The proposed method has several advantageous features:
  \begin{inparaenum}[(i)]
  \item It supports fairly general meshes possibly containing polyhedral elements and nonmatching interfaces;
  \item it allows arbitrary approximation orders;
  \item it has a moderate computational cost thanks to the possibility of locally eliminating element-based unknowns by static condensation.
  \end{inparaenum}
  We perform a detailed stability and convergence study, proving optimal convergence rates in energy-like norms.
  Numerical validation is also provided using some of the most common tests in the literature.
\end{abstract}

\section{Introduction}

Let $\Omega\subset\Real^d$, $d\in\{2,3\}$, denote a bounded connected convex polyhedral domain with boundary $\partial\Omega$ and outward normal $\normal$, and let $\tF>0$.
The Cahn--Hilliard problem, originally introduced in~\cite{Cahn.Hilliard:58,Cahn:61} to model phase separation in a binary alloy, consists in finding the order-parameter $c:\Omega\times\lbrack 0,\tF\rbrack\to\Real$ and chemical potential $w:\Omega\times\lbrack 0,\tF\rbrack\to\Real$ such that
\begin{subequations}
  \label{eq:strong}
  \begin{alignat}{2}
    \label{eq:strong:1}
    d_t c - \LAPL w &= 0 &\qquad&\text{in $\Omega\times(0,\tF\rbrack$},
    \\
    \label{eq:strong:2}
    w &= \Phi'(c) - \gamma^2\LAPL c &\qquad&\text{in $\Omega\times(0,\tF\rbrack$},
    \\
    \label{eq:strong:ic}
    c(0) &= c_0 &\qquad&\text{in $\Omega$},
    \\
    \label{eq:strong:bc}
    \partial_{\normal}c=\partial_{\normal}w&=0 &\qquad&\text{on $\partial\Omega\times (0,\tF\rbrack$},
  \end{alignat}
\end{subequations}
where $c_0\in H^2(\Omega)\cap L^2_0(\Omega)$ such that $\partial_{\normal} c_0=0$ on $\partial\Omega$ denotes the initial datum, $\gamma>0$ the interface parameter (usually taking small values), and $\Phi$ the free-energy such that
\begin{equation}
  \label{eq:Phi}
  \Phi(c)\eqbydef\frac14(1-c^2)^2.
\end{equation}
Relevant extensions of problem~\eqref{eq:strong} (not considered here) include the introduction of a flow which requires, in particular, to add a convective term in~\eqref{eq:strong:1}; cf., e.g.,~\cite{Jacqmin:99,Badalassi.Ceniceros.ea:03,Boyer:02,Boyer.Chupin.ea:04,Kim.Kang.ea:04,Kay.Styles.ea:09}.

The discretization of the Cahn--Hilliard equation~\eqref{eq:strong} has been considered in several works.
Different aspects of standard finite element schemes have been studied, e.g., in~\cite{Elliott.French.ea:89,Du.Nicolaides:91,Copetti.Elliott:92}; cf. also the references therein.
Mixed finite elements are considered in~\cite{Feng.Prohl:05}.
In~\cite{Wells.Kuhl.ea:06}, the authors study a nonconforming method based on $C^0$ shape functions for the fourth-order primal problem obtained by plugging~\eqref{eq:strong:2} into~\eqref{eq:strong:1}.
Discontinuous Galerkin (dG) methods have also received extensive attention.
We can cite here~\cite{Xia.Xu.ea:07}, where a local dG method is proposed for a Cahn--Hilliard system modelling multi-component alloys, and a stability analysis is carried out;
\cite{Feng.Karakashian:07}, where optimal error estimates are proved for a dG discretization of the Cahn--Hilliard problem in primal form;
\cite{Kay.Styles.ea:09}, which contains optimal error estimates for a dG method based on the mixed formulation of the problem including a convection term;
~\cite{Guo.Xu:14}, where a multi-grid approach is proposed for the solution of the systems of algebraic equations arising from a dG discretization of the Cahn--Hilliard equation.
In all of the above references, standard meshes are considered.
General polygonal meshes in dimension $d=2$, on the other hand, are supported by the recently proposed $C^1$-conforming Virtual Element (VE) method of~\cite{Antonietti.Beirao-da-Veiga.ea:15} for the problem in primal formulation; cf. also~\cite{Beirao-da-Veiga.Manzini:14} for VE methods with arbitrary regularity.
Therein, the convergence analysis is carried out under the assumption that the discrete order-parameter satisfies a $C^0(L^\infty)$-like a priori bound.

In this work, we develop and analyze a fully implicit Hybrid High-Order (HHO) algorithm for problem~\eqref{eq:strong} where the space discretization is based on the HHO($k+1)$ variation proposed in~\cite{Cockburn.Di-Pietro.ea:15} of the method of~\cite{Di-Pietro.Ern.ea:14}.
The method hinges on hybrid degrees of freedom (DOFs) located at mesh elements and faces that are polynomials of degree $(k+1)$ and $k$, respectively.
The nonlinear term in~\eqref{eq:strong:2} is discretized by means of element unknowns only.
For the second-order diffusive operators in~\eqref{eq:strong:1} and~\eqref{eq:strong:2}, on the other hand, we rely on two key ingredients devised locally inside each element:
\begin{inparaenum}[(i)]
\item A potential reconstruction obtained from the solution of (small) Neumann problems and
\item a stabilization term penalizing the lowest-order part of the difference between element- and face-based unknowns.
\end{inparaenum}
See also~\cite{Cockburn.Gopalakrishnan.ea:09,Wang.Ye:13,Lipnikov.Manzini:14} for related methods for second-order linear diffusion operators, each displaying a set of distinctive features.
The global discrete problem is then obtained by a standard element-by-element assembly procedure.
When using a first-order (Newton-like) algorithm to solve the resulting system of nonlinear algebraic equations, element-based unknowns can be statically condensed.
As a result, the only globally coupled unknowns in the linear subproblems are discontinuous polynomials of degree $k$ on the mesh skeleton for both the order-parameter and the chemical potential.
With a backward Euler scheme to march in time, the $C^0(H^1)$-like error on the order-parameter and the $L^2(H^1)$-like error on the chemical potential are proved to optimally converge as $(h^{k+1}+\tau)$ (with $h$ and $\tau$ denoting, respectively, the spatial and temporal mesh sizes) provided the solution has sufficient regularity.

The proposed method has several advantageous features:
\begin{inparaenum}[(i)]
\item It supports general meshes possibly including polyhedral elements and nonmatching interfaces (resulting, e.g., from nonconforming mesh refinement);
\item it allows one to increase the spatial approximation order to accelerate convergence in the presence of (locally) regular solutions;
\item it is (relatively) inexpensive. When $d=2$, e.g., the number of globally coupled spatial unknowns for our method scales as $2\card{\Fh}(k+1)$ (with $\card{\Fh}$ denoting the number of mesh faces) as opposed to $\card{\Th}(k+3)(k+2)$ (with $\card{\Th}$ denoting the number of mesh elements) for a mixed dG method delivering the same order of convergence (i.e., based on broken polynomials of degree $k+1$).
\end{inparaenum}
Additionally, thanks to the underlying fully discontinuous polynomial spaces, the proposed method can accomodate abrupt variations of the unknowns in the vicinity of the interface between phases.

Our analysis adapts the techniques originally developed in~\cite{Kay.Styles.ea:09} in the context of dG methods.
Therein, the treatment of the nonlinear term in~\eqref{eq:strong:2} hinges on $C^0$-in-time a priori estimates for various norms and seminorms of the discrete order-parameter.
Instrumental in proving these estimates are discrete functional analysis results, including discrete versions of Agmon's and Gagliardo--Nirenberg--Poincar\'{e}'s inequalities for broken polynomial functions on quasi-uniform matching simplicial meshes.
Adapting these tools to hybrid polynomial spaces on general meshes entails several new ideas.
A first key point consists in defining appropriate discrete counterparts of the Laplace and Green's operators.
To this end, we rely on a suitably tailored $L^2$-like hybrid inner product which guarantees stability estimates for the former and optimal approximation properties for the latter.
Another key point consists in replacing the standard nodal interpolator used in the proofs of \cite[Lemmas~2.2 and 2.3]{Kay.Styles.ea:09} by the $L^2$-orthogonal projector which, unlike the former, is naturally defined for meshes containing polyhedral elements.
We show that this replacement is possible thanks to the $W^{s,p}$-stability and approximation properties of the $L^2$-orthogonal projector on broken polynomial spaces recently presented in a unified setting in~\cite{Di-Pietro.Droniou:15}; cf. also the references therein for previous results on this subject.

The material is organized as follows:
In Section~\ref{sec:disc} we introduce the notation for space and time meshes and recall some key results on broken polynomial spaces;
in Section~\ref{sec:hho} we introduce hybrid polynomial spaces and local reconstructions, and state the discrete problem;
in Section~\ref{sec:stab} we carry out the stability analysis of the method, while the convergence analysis is detailed in Section~\ref{sec:err.anal};
Section~\ref{sec:num.tests} contains an extensive numerical validation of the proposed algorithm;
finally, in Appendix~\ref{sec:proofs} we give proofs of the discrete functional analysis results used to derive stability bounds and error estimates.

%------------------------------------------------------------------------------%

\section{Discrete setting}\label{sec:disc}

In this section we introduce the discrete setting and recall some basic results on broken polynomial spaces.

\subsection{Space and time meshes}\label{sec:setting}

We recall here the notion of admissible spatial mesh sequence from~\cite[Chapter~1]{Di-Pietro.Ern:12}.
For the sake of simplicity, we will systematically use the term polyhedral also when $d=2$.
Denote by ${\cal H}\subset \Real_*^+ $ a countable set of spatial meshsizes having $0$ as its unique accumulation point.
We consider $h$-refined mesh sequences $(\Th)_{h \in {\cal H}}$ where, for all $ h \in {\cal H} $, $\Th$ is a finite collection of nonempty disjoint open polyhedral elements $T$ of boundary $\partial T$ such that $\closure{\Omega}=\bigcup_{T\in\Th}\closure{T}$ and $h=\max_{T\in\Th} h_T$ with $h_T$ standing for the diameter of the element $T$.

A face $F$ is defined as a planar closed connected subset of $\closure{\Omega}$ with positive $ (d{-}1) $-dimensional Hausdorff measure and such that%
\begin{inparaenum}[(i)]
\item either there exist $T_1,T_2\in\Th $ such that $F\subset\partial T_1\cap\partial T_2$ and $F$ is called an interface or 
\item there exists $T\in\Th$ such that $F\subset\partial T \cap\partial\Omega$ and $F$ is called a boundary face.
\end{inparaenum}
Mesh faces are collected in the set $\Fh$, and the diameter of a face $F\in\Fh$ is denoted by $h_F$.
For all $T\in\Th$, $\Fh[T]\eqbydef\{F\in\Fh\st F\subset\partial T \}$ denotes the set of faces lying on $\partial T$ and, for all $F\in\Fh[T]$, $\normal_{TF}$ is the unit normal to $F$ pointing out of $T$.
Symmetrically, for all $F\in\Fh$, we denote by $\Th[F]$ the set of one (if $F\in\Fhb$) or two (if $F\in\Fhi$) elements sharing $F$.

\begin{assumption}[Admissible spatial mesh sequence]\label{ass:mesh}
  We assume that, for all $h\in{\cal H}$, $\Th$ admits a matching simplicial submesh $\fTh$ and there exists a real number $\varrho>0$ independent of $h$ such that, for all $h\in{\cal H}$, the following properties hold:%
  \begin{inparaenum}[(i)]
  \item {\em Shape regularity:} For all simplex $S\in\fTh$ of diameter $h_S$ and inradius $r_S$, $\varrho h_S\le r_S$;
  \item {\em contact-regularity:} For all $T\in\Th$, and all $S\in\fTh$ such that $S\subset T$, $\varrho h_T \le h_S$.
  \end{inparaenum}
\end{assumption}

To discretize in time, we consider a uniform partition $(t^n)_{0\le n\le N}$ of the time interval $[0,\tF]$ with $t^0=0$, $t^N=\tF$ and $t^n-t^{n-1}=\tau$ for all $1\le n\le N$ (the analysis can be adapted to nonuniform partitions).
For any sufficiently regular function of time $\varphi$ taking values in a vector space $V$, we denote by $\varphi^n\in V$ its value at discrete time $t^n$, and we introduce the backward differencing operator $\dt$ such that, for all $1\le n\le N$,
\begin{equation}
  \label{eq:ddt}
  \dt\varphi^n\eqbydef\frac{\varphi^n-\varphi^{n-1}}{\tau}\in V.
\end{equation}

In what follows, we often abbreviate by $a\lesssim b$ the inequality $a\le Cb$ with $a$ and $b$ positive real numbers and $C>0$ generic constant independent of both the meshsize $h$ and the time step $\tau$ (named constants are used in the statements for the sake of easy consultation).
Also, for a subset $X\subset\closure{\Omega}$, we denote by $(\cdot,\cdot)_X$ and $\norm[X]{{\cdot}}$ the usual $L^2(X)$-inner product and norm, with the convention that we omit the index if $X=\Omega$. The same notation is used for the vector-valued space $L^2(X)^d$.

\subsection{Basic results on broken polynomial spaces}

The proposed method is based on local polynomial spaces on mesh elements and faces.
Let an integer $l\ge 0$ be fixed. Let $U$ be a subset of $\Real^d$, $H_U$ the affine space spanned by $U$, $d_U$ its dimension, and assume that $U$ has a non-empty interior in $H_U$.
We denote by $\Poly{l}(U)$ the space spanned by $d_U$-variate polynomials on $H_U$ of total degree $l$, and by $\lproj[U]{l}$ the $L^2$-orthogonal projector onto this space.
In the following sections, the set $U$ will represent a mesh element or face.
The space of broken polynomial functions on $\Th$ of degree $l$ is denoted by $\Poly{l}(\Th)$, and $\lproj{l}$ is the corresponding $L^2$-orthogonal projector.

We next recall some functional analysis results on polynomial spaces.
The following discrete trace and inverse inequalities are proved in~\cite[Chapter~1]{Di-Pietro.Ern:12} (cf. in particular Lemmas~1.44 and 1.46):
There is $C>0$ independent of $h$ such that, for all $T\in\Th$, and all $\forall v\in\Poly{l}(T)$,
\begin{equation}
  \label{eq:trace.disc}
  \norm[F]{v} \le C h_F^{- \frac12} \norm[T]{v}\qquad\forall F\in\Fh[T],
\end{equation}
and
\begin{equation}
  \label{eq:inv}
  \norm[T]{\GRAD v}\le C h_T^{-1}\norm[T]{v}.
\end{equation}
We will also need the following local direct and reverse Lebesgue embeddings (cf.~\cite[Lemma~5.1]{Di-Pietro.Droniou:15}):
There is $C>0$ independent of $h$ such that, for all $T\in\Th$, all $q,p\in[1,+\infty]$, 
\begin{equation}\label{eq:leb.emb}
  \forall v\in\Poly{l}(T),\qquad
  C^{-1}\norm[L^q(T)]{v}\le h_T^{\frac{d}{q}-\frac{d}{p}}\norm[L^p(T)]{v}\le C\norm[L^q(T)]{v}.
\end{equation}

The proof of the following results for the local $L^2$-orthogonal projector can be found in~\cite[Appendix A.2]{Di-Pietro.Droniou:15}.
For an open set $U$ of $\Real^d$, $s\in\Natural$ and $p\in [1,+\infty]$, we define the seminorm $\seminorm[W^{s,p}(U)]{{\cdot}}$ as follows: For all $v\in W^{s,p}(U)$, $$\seminorm[W^{s,p}(U)]{v}\eqbydef
\sum_{\vec{\alpha}\in\Natural^d,\,|\vec{\alpha}|_{\ell^1}=s}\norm[L^p(U)]{\partial^{\vec{\alpha}} v},$$ where $|\vec{\alpha}|_{\ell^1}\eqbydef\alpha_1+\cdots+\alpha_d$ and $\partial^{\vec{\alpha}}=\partial_1^{\alpha_1}\cdots\partial_d^{\alpha_d}$.
For $s=0$, we recover the usual Lebesgue spaces $L^p(U)$.
The $L^2$-orthogonal projector is $W^{s,p}$-stable and has optimal $W^{s,p}$-approximation properties:
There is $C>0$ independent of $h$ such that, for all $T\in\Th$, all $s\in\{0,\ldots,l+1\}$, all $p\in[1,+\infty]$, and all $v\in W^{s,p}(T)$, it holds,
\begin{equation}\label{eq:lproj.stab}
  \seminorm[W^{s,p}(T)]{\lproj[T]{l}v}\le C \seminorm[W^{s,p}(T)]{v},
\end{equation}
and, for all $m\in\{0,\ldots,s\}$,
\begin{equation}\label{eq:lproj.approx}
  \seminorm[W^{m,p}(T)]{v-\lproj[T]{l} v}
  + h_T^{\frac1p}\seminorm[{W^{m,p}(\Fh[T])}]{v-\lproj[T]{l} v}
  \le C h_T^{s-m}\seminorm[W^{s,p}(T)]{v},
\end{equation}
where $W^{m,p}(\Fh[T])$ denotes the set of functions that belong to $W^{m,p}(F)$ for all $F\in\Fh[T]$.
Finally, there is $C>0$ independent of $h$ such that it holds, for all $F\in\Fh$,
\begin{equation}\label{eq:lprojF.approx}
  \forall v\in H^1(F),\qquad
  \norm[F]{v-\lproj[F]{l}v}\le C h\seminorm[H^1(F)]{v}.
\end{equation}%

In the proofs of Lemmas~\ref{lem:agmon} and~\ref{lem:gnp} below, we will make use of the following global inverse inequalities, which require mesh quasi-uniformity.

\begin{proposition}[Global inverse inequalities for Lebesgue norms of broken polynomials]
  In addition to Assumption~\ref{ass:mesh}, we assume that the mesh is quasi-uniform, i.e.,
  \begin{equation}\label{eq:qu}
    \forall T\in\Th,\qquad \varrho h\le h_T.
  \end{equation}
  Then, for all polynomial degree $l\ge 0$ and all $1\le p\le q\le+\infty$, it holds
  \begin{equation}\label{eq:glob.inv}
    \forall w_h\in\Poly{l}(\Th),\qquad
    \norm[L^q(\Omega)]{w_h}\le C h^{\frac{d}{q} - \frac{d}{p}}\norm[L^p(\Omega)]{w_h},
  \end{equation}
  with real number $C>0$ independent of $h$.
\end{proposition}

\begin{proof}
  Let $w_h\in\Poly{l}(\Th)$.
  We start by proving that, for all $p\in[1,+\infty]$,
  \begin{equation}\label{eq:glob.inv.Linfty}
    \forall w_h\in\Poly{l}(\Th),\qquad
    \norm[L^\infty(\Omega)]{w_h}\lesssim h^{-\frac{d}{p}}\norm[L^p(\Omega)]{w_h},
  \end{equation}
  which corresponds to~\eqref{eq:glob.inv} with $q=+\infty$.
  By the local reverse Lebesgue embeddings~\eqref{eq:leb.emb}, there is $C>0$ independent of $h$ such that
  $$
  \forall T\in\Th,\qquad\norm[L^\infty(T)]{w_h}\le C h_T^{-\frac{d}{p}}\norm[L^p(T)]{w_h}\le C \rho^{-\frac{d}{p}} h^{-\frac{d}{p}}\norm[L^p(\Omega)]{w_h},
  $$
  where we have used the mesh quasi-uniformity assumption~\eqref{eq:qu} to conclude.
  Inequality~\eqref{eq:glob.inv.Linfty} follows observing that $\norm[L^\infty(\Omega)]{w_h}=\max_{T\in\Th}\norm[L^\infty(T)]{w_h}$.
  Let us now turn to the case $1\le q<+\infty$. We have
  $$
  \norm[L^q(\Omega)]{w_h}^q
  \le\norm[L^\infty(\Omega)]{w_h}^{q-p}\norm[L^p(\Omega)]{w_h}^p
  \lesssim \left(h^{\frac{d}{q}-\frac{d}{p}}\norm[L^p(\Omega)]{w_h}\right)^{q},
  $$
  where the conclusion follows using~\eqref{eq:glob.inv.Linfty}.
\end{proof}

%------------------------------------------------------------------------------%

\section{The Hybrid High-Order method}\label{sec:hho}

In this section we define hybrid spaces and state the discrete problem.

\subsection{Hybrid spaces}\label{sec:setting:spaces}

The discretization of the diffusion operator hinges on the HHO method of~\cite{Cockburn.Di-Pietro.ea:15} using polynomials of degree $(k+1)$ inside elements and $k$ on mesh faces (cf. Remark~\ref{rem:need.k+1} for further insight on this choice).
The global discrete space is defined as
\begin{equation}\label{eq:Uh}
  \Uh\eqbydef\left(
  \bigtimes_{T\in\Th}\Poly{k+1}(T)
  \right)\times\left(
  \bigtimes_{F\in\Fh}\Poly{k}(F)
  \right).
\end{equation}
The restriction of $\Uh$ to an element $T\in\Th$ is denoted by $\UT$.
For a generic collection of DOFs in $\Uh$, we use the underlined notation $\uvh=((v_T)_{T\in\Th},(v_F)_{F\in\Fh})$ and, for all $T\in\Th$, we denote by $\uv=(v_T,(v_F)_{F\in\Fh[T]})$ its restriction to $\UT$.
Also, to keep the notation compact, we denote by $v_h$ (no underline) the function in $\Poly{k+1}(\Th)$ such that
\begin{equation*}\label{eq:vh}
  \restrto{v_h}{T}=v_T\qquad\forall T\in\Th.
\end{equation*}
In what follows, we will also need the zero-average subspace 
$$
\UhO\eqbydef\left\{ \uvh\in\Uh\st (v_h,1)=0 \right\}.
$$
The interpolator $\Ih:H^1(\Omega)\to\Uh$ is such that, for all $v\in H^1(\Omega)$,
\begin{equation}\label{eq:Ih}
  \Ih v\eqbydef ( (\lproj[T]{k+1}v)_{T\in\Th}, (\lproj[F]{k}v)_{F\in\Fh}).
\end{equation}
We define on $\Uh$ the seminorm $\norm[1,h]{{\cdot}}$ such that
\begin{equation}
  \label{eq:norm1h}
  \norm[1,h]{\uvh}^2\eqbydef\norm{\GRADh v_h}^{2} + \seminorm[1,h]{\uvh}^2,\qquad
  \seminorm[1,h]{\uvh}^2\eqbydef s_{1,h}(\uvh,\uvh),
\end{equation}
where $\GRADh$ denotes the usual broken gradient on $H^1(\Th)$ and the stabilization bilinear form $s_{1,h}$ on $\Uh\times\Uh$ is such that
\begin{equation}\label{eq:s1h}
  s_{1,h}(\uvh,\uzh)\eqbydef\sum_{T\in\Th}\sum_{F\in\Fh[T]}h_F^{-1}(\lproj[F]{k}(v_F-v_T),\lproj[F]{k}(z_F-z_T))_F.
\end{equation}
Using the stability and approximation properties of the $L^2$-orthogonal projector expressed by~\eqref{eq:lproj.stab}--\eqref{eq:lproj.approx}, one can prove that $\Ih$ is $H^1$-stable:
\begin{equation}\label{eq:Ih.stab}
  \forall v\in H^1(\Omega),\qquad
  \norm[1,h]{\Ih v}\lesssim\norm[H^1(\Omega)]{v}.
\end{equation}

The following Friedrichs' inequalities can be proved using the arguments of~\cite[Lemma~7.2]{Di-Pietro.Droniou:15}, where element DOFs of degree $k$ are considered (cf. also~\cite{Buffa.Ortner:09,Di-Pietro.Ern:10} for related results using dG norms):
For all $r\in\lbrack 1,+\infty)$ if $d=2$, all $r\in\lbrack 1,6\rbrack$ if $d=3$,
\begin{equation}
  \label{eq:friedrichs}
  \forall\uvh\in\UhO,\qquad
  \norm[L^r(\Omega)]{v_h}\lesssim\norm[1,h]{\uvh}.
\end{equation}
The case $r=2$ corresponds to Poincar\'{e}'s inequality.
Finally, to close this section, we prove that $\norm[1,h]{{\cdot}}$ defines a norm on $\UhO$.

\begin{proposition}[{Norm $\norm[1,h]{{\cdot}}$}]\label{prop:norm1h}
  The map $\norm[1,h]{{\cdot}}$ defines a norm on $\UhO$.
\end{proposition}

\begin{proof}
  We only have to show that $\norm[1,h]{\uvh}=0\implies\uvh=\underline{0}$.
  By~\eqref{eq:friedrichs}, $\norm[1,h]{\uvh}\implies v_h\equiv 0$, i.e., $v_T\equiv 0$ for all $T\in\Th$.
  Plugging this result into the definition~\eqref{eq:norm1h} of $\norm[1,h]{{\cdot}}$, we get $\sum_{T\in\Th}\sum_{F\in\Fh[T]} h_F^{-1}\norm[F]{v_F}^2=0$, which implies in turn $v_F\equiv 0$ for all $F\in\Fh$.
\end{proof}

\subsection{Diffusive bilinear form and discrete problem}\label{sec:setting:diffusion}

For all $T\in\Th$, we define the potential reconstruction operator $\pT:\UT\to\Poly{k+1}(T)$ such that, for all $\uv\in\UT$, $\pT\uv$ is the unique solution of the following Neumann problem:
\begin{equation}\label{eq:pT}
  (\GRAD\pT\uv,\GRAD z)_T
  = -( v_T,\LAPL z)_T + \sum_{F\in\Fh[T]}(v_F, \GRAD z\SCAL\normal_{TF})_F\qquad
  \forall z\in\Poly{k+1}(T),
\end{equation}
with closure condition $(\pT\uv,1)_T=(v_T,1)_T$.
It can be proved that, for all $v\in H^1(T)$, denoting by $\IT$ the restriction of the reduction map $\Ih$ defined by~\eqref{eq:Ih} to $H^1(T)\to\UT$,
\begin{equation}\label{eq:pT.IT.ell.proj}
  (\GRAD(\pT\IT v - v),\GRAD z)_T = 0\qquad\forall z\in\Poly{k+1}(T),
\end{equation}
which expresses the fact that $(\pT\circ\IT)$ is the elliptic projector onto $\Poly{k+1}(T)$ (and, as such, has optimal approximation properties in $\Poly{k+1}(T)$).
The diffusive bilinear form $a_h$ on $\Uh\times\Uh$ is obtained by element-wise assembly setting
\begin{equation}\label{eq:ah}
  a_h(\uvh,\uzh)\eqbydef\sum_{T\in\Th}(\GRAD\pT\uv,\GRAD\pT\uz)_T + s_{1,h}(\uvh,\uzh),
\end{equation}
with stabilization bilinear form $s_{1,h}$ defined by~\eqref{eq:s1h}.
Denoting by $\norm[a,h]{{\cdot}}$ the seminorm defined by $a_h$ on $\Uh$, a straightforward adaptation of the arguments used in~\cite[Lemma~4]{Di-Pietro.Ern.ea:14} shows that
\begin{equation}\label{eq:norm1h.ah}
  \forall\uvh\in\Uh,\qquad
  \norm[1,h]{\uvh}\lesssim\norm[a,h]{\uvh}\lesssim\norm[1,h]{\uvh},
\end{equation}
which expresses the coercivity and boundedness of $a_h$.
Additionally, following the arguments in~\cite[Theorem~8]{Di-Pietro.Ern.ea:14}, one can easily prove that the bilinear form $a_h$ enjoys the following consistency property:
For all $v\in H^{\max(2,l)}(\Omega)\cap L^2_0(\Omega)$, $l\ge 1$, such that $\partial_{\normal} v=0$ on $\partial\Omega$,
\begin{equation}\label{eq:cons.ah}
  \sup_{\uzh\in\UhO,\norm[1,h]{\uzh}=1}\left|a_h(\Ih v,\uzh) + (\LAPL v, z_h)\right|
  \lesssim h^{\min(k+1,l-1)}\norm[H^l(\Omega)]{v}.
\end{equation}

\begin{remark}[Consistency of $a_{h}$]
  For sufficiently regular solutions (i.e., when \mbox{$l=k+2$)}, equation~\eqref{eq:cons.ah} shows that the consistency error scales as $h^{k+1}$.
  This is a consequence of the fact that both the potential reconstruction $\pT$ (cf.~\eqref{eq:pT}) and the stabilization bilinear form $s_{1,h}$ (cf.~\eqref{eq:s1h}) are consistent for exact solutions that are polynomials of degree $(k+1)$ inside each element.
  In particular, a key point in $s_{1,h}$ is to penalize $\lproj[F]{k}(v_F-v_T)$ instead of $(v_F-v_T)$.
  A similar stabilization bilinear form had been independently suggested in the context of Hybridizable Discontinuous Galerkin methods in~\cite[Remark~1.2.4]{Lehrenfeld:10}.
\end{remark}

The discrete problem reads:
For all $1\le n\le N$, find $(\uchn,\uwhn)\in\UhO\times\Uh$ such that
\begin{subequations}
  \label{eq:discrete}
  \begin{alignat}{2}
    \label{eq:discrete:1}
    &(\dt c_h^n, \varphi_h) + a_h(\uwhn,\uphi) = 0 &\qquad&\forall\uphi\in\Uh,
    \\
    \label{eq:discrete:2}
    &(w_h^n,\psi_h) = (\Phi'(c_h^n),\psi_h) + \gamma^2 a_h(\uchn,\upsi) &\qquad&\forall\upsi\in\Uh,
  \end{alignat}
\end{subequations}
and $\uchn[0]\in\UhO$ solves
\begin{equation}\label{eq:discrete.ic}
  a_h(\uchn[0],\uphi) = -(\LAPL c_0,\varphi_h)\qquad\forall\uphi\in\Uh.
\end{equation}
We note, in passing, that the face DOFs in $\uchn[0]$ are not needed to initialize the algorithm.

\begin{remark}[Static condensation]\label{rem:stat.cond}
  Problem~\eqref{eq:discrete} is a system of nonlinear algebraic equations, which can be solved using an iterative algorithm.
  When first order (Newton-like) algorithms are used, element-based DOFs can be locally eliminated at each iteration by a standard static condensation procedure.
\end{remark}

%------------------------------------------------------------------------------%

\section{Stability analysis}\label{sec:stab}

In this section we establish some uniform a priori bounds on the discrete solution.
To this end, we need a discrete counterpart of Agmon's inequality; cf.~\cite[Lemma~13.2]{Agmon:10} and also~\cite[Theorem~3]{Adams.Fournier:77}.
We define on $\Uh$ the following $L^2$-like inner product:
\begin{equation}\label{eq:prod0h}
  \begin{aligned}
    (\uvh,\uzh)_{0,h}&\eqbydef (v_h,z_h) + s_{0,h}(\uvh,\uzh),
    \\
    s_{0,h}(\uvh,\uzh)&\eqbydef \sum_{T\in\Th}\sum_{F\in\Fh[T]} h_F(\lproj[F]{k}(v_F-v_T),\lproj[F]{k}(z_F-z_T))_F,
  \end{aligned}
\end{equation}
and denote by $\norm[0,h]{{\cdot}}$ and $\seminorm[0,h]{{\cdot}}$ the norm and seminorm corresponding to the bilinear forms $(\cdot,\cdot)_{0,h}$ and $s_{0,h}$, respectively.
For further insight on the role of $s_{0,h}$, cf. Remark~\ref{rem:s0h}.
We introduce the discrete Laplace operator $\uLh:\Uh\to\Uh$ such that, for all $\uvh\in\Uh$,
\begin{equation}
  \label{eq:uLh}
  -(\uLh\uvh,\uzh)_{0,h} = a_h(\uvh,\uzh)\qquad\forall\uzh\in\Uh,
\end{equation}
and we denote by $\Lh\uvh$ (no underline) the broken polynomial function in $\Poly{k+1}(\Th)$ obtained from element DOFs in $\uLh\uvh$.

\begin{remark}[Restriction of $\uLh$ to $\UhO\to\UhO$]\label{rem:restr.uLh}
  Whenever $\uvh\in\UhO$, $\uLh\uvh\in\UhO$.
  To prove it, it suffices to take $\uzh=\Ih\chfun{\Omega}$ in~\eqref{eq:uLh} (with $\chfun{\Omega}$ characteristic function of $\Omega$), and observe that the left-hand side satisfies $(\uLh\uvh,\uzh)_{0,h}=(\Lh\uvh,1)$ while, by definition~\eqref{eq:ah} of the bilinear form $a_h$, the right-hand side vanishes.
  In what follows, we keep the same notation for the (bijective) restriction of $\uLh$ to $\UhO\to\UhO$.
\end{remark}

The following result, valid for $d\in\{2,3\}$, will be proved in Appendix~\ref{sec:proofs}.
\begin{lemma}[Discrete Agmon's inequality]\label{lem:agmon}
  Assume mesh quasi-uniformity~\eqref{eq:qu}.
  Then, it holds with real number $C>0$ independent of $h$,
  \begin{equation}\label{eq:agmon}
    \forall\uvh\in\UhO,\qquad
    \norm[L^\infty(\Omega)]{v_h}\le C
    \norm[1,h]{\uvh}^{\frac12}\norm[0,h]{\uLh\uvh}^{\frac12}.
  \end{equation}
\end{lemma}

We also recall the following discrete Gronwall's inequality (cf.~\cite[Lemma~5.1]{Heywood.Rannacher:90}).

\begin{lemma}[Discrete Gronwall's inequality]\label{lem:disc.gronwall}
  Let two reals $\delta,G>0$ be given, and, for integers $n\ge 1$, let $a^n$, $b^n$, and $\chi^n$ denote nonnegative real numbers such that
  $$
  a^N + \delta\sum_{n=1}^N b^n \le
  \delta\sum_{n=1}^N\chi^na^n + G\qquad\forall N\in\Natural^*.
  $$
  Then, if $\chi^n\delta<1$ for all $n$, letting $\varsigma^n\eqbydef (1-\chi^n\delta)^{-1}$, it holds
  \begin{equation}
    \label{eq:disc.gronwall}
    a^N + \delta\sum_{n=1}^N b^n \le \exp\left(
    \delta\sum_{n=1}^N\varsigma^n\chi^n
    \right)\times G\qquad\forall N\in\Natural^*.
  \end{equation}
\end{lemma}

We are now ready to prove the a priori bounds.

\begin{lemma}[Uniform a priori bounds]\label{lem:a-priori}
  Under the assumptions of Lemma~\ref{lem:agmon}, and further assuming that $\tau\le L$ for a given real number $L>0$ independent of $h$ and of $\tau$ (but depending on $\gamma^2$) and sufficiently small, there is a real number $C>0$ independent of $h$ and $\tau$ such that
  \begin{equation*}\label{eq:a-priori}
    \max_{1\le n\le N}\left(
    \norm[a,h]{\uchn}^2 
    + (\Phi(c_h^n),1)
    + \norm{w_h^n}^2
    + \norm[L^\infty(\Omega)]{c_h^n}
    + \norm[0,h]{\uLh\uchn}^2
    \right) + \sum_{n=1}^N\tau\norm[a,h]{\uwhn}^2
    \le C.
  \end{equation*}
\end{lemma}

\begin{proof}
  The proof is split into several steps.
  \begin{asparaenum}[(i)]
  \item We start by proving that
    \begin{equation}
      \label{eq:a-priori:i}
      \max_{1\le n\le N}\left(
      \norm[a,h]{\uchn}^2    
      + (\Phi(c_h^n),1)
      \right) + \sum_{n=1}^N\tau\norm[a,h]{\uwhn}^2
      \lesssim 1.
    \end{equation}
    Subtracting~\eqref{eq:discrete:2} with $\upsi=\uchn-\uchn[n-1]$ from~\eqref{eq:discrete:1} with $\uphi=\tau\uwhn$, and using the fact that, for all $r,s\in\Real$, $\Phi'(r)(r-s)\ge\Phi(r)-\Phi(s)-\frac12(r-s)^2$, it is inferred, for all $1\le n\le N$, that
    \begin{equation}\label{eq:a-priori:1}
      \gamma^2a_h(\uchn,\uchn-\uchn[n-1])
      + \tau\norm[a,h]{\uwhn}^2 
      + (\Phi(c_h^n),1)
      \le
      \frac12\norm{c_h^n-c_h^{n-1}}^2
      + (\Phi(c_h^{n-1}),1).
    \end{equation}
    Notice that $(\Phi(c_h^n),1)\ge 0$ for all $0\le n\le N$ by definition~\eqref{eq:Phi} of $\Phi$.
    Making $\uphi=\tau(\uchn-\uchn[n-1])$ in~\eqref{eq:discrete:1} and using the Cauchy--Schwarz and Young's inequalities, we infer that
    \begin{equation}\label{eq:a-priori:2}
      \norm{c_h^n-c_h^{n-1}}^2
      \le\frac{\tau}{2}\norm[a,h]{\uwhn}^2
      + \frac{\tau}{2}\norm[a,h]{\uchn-\uchn[n-1]}^2.
    \end{equation}
    Additionally, recalling the following formula for the backward Euler scheme:
    \begin{equation}\label{eq:magic:BE}
      2x(x-y) = x^2 + (x-y)^2 - y^2,
    \end{equation}
    it holds
    \begin{equation}\label{eq:a-priori:3}
      a_h(\uchn,\uchn-\uchn[n-1])
      = \frac12\left(
      \norm[a,h]{\uchn}^2
      + \norm[a,h]{\uchn-\uchn[n-1]}^2
      - \norm[a,h]{\uchn[n-1]}^2
      \right).
    \end{equation}
    Plugging~\eqref{eq:a-priori:2} and~\eqref{eq:a-priori:3} into~\eqref{eq:a-priori:1}, we obtain
    \begin{multline*}
    \gamma^2\norm[a,h]{\uchn}^2
    + \left(\gamma^2-\frac{\tau}{2}\right)\norm[a,h]{\uchn-\uchn[n-1]}^2
    + \frac{3\tau}{2}\norm[a,h]{\uwhn}^2
    + 2(\Phi(c_h^n),1)
    \\ \le
    \gamma^2\norm[a,h]{\uchn[n-1]}^2
    + 2(\Phi(c_h^{n-1}),1).
    \end{multline*}
    Provided $\tau<2\gamma^2$, the bound~\eqref{eq:a-priori:i} follows summing the above inequality over $1\le n\le N$, and using the fact that $\gamma^2\norm[a,h]{c_h^0}+2(\Phi(c_h^0),1)\lesssim 1$.
    To prove this bound, observe that
    $$
    \begin{aligned}
      \gamma^2\norm[a,h]{c_h^0}+2(\Phi(c_h^0),1)
      &\lesssim \gamma^2\norm[1,h]{\uchn[0]}^2 + 1 + \norm[L^4(\Omega)]{c_h^0}^4 + \norm{c_h^0}^2 
      \\
      &\lesssim \gamma^2\norm[1,h]{\uchn[0]}^2 + 1 + \norm[1,h]{\uchn[0]}^4 + \norm[1,h]{\uchn[0]}^2
      \lesssim 1,
    \end{aligned}
    $$
    where we have used the definition~\eqref{eq:Phi} of the free-energy $\Phi$ in the first line followed by the discrete Friedrichs' inequality with $r=4,2$ in the second line and the first bound on the initial datum in~\eqref{eq:bnd.uchn0} below to conclude.
    
  \item We next prove that 
    \begin{equation}\label{eq:a-priori:ii}
      \sum_{n=1}^N\tau\norm[L^\infty(\Omega)]{c_h^n}^4\lesssim 1.
    \end{equation}
    The discrete Agmon's inequality~\eqref{eq:agmon} followed by the first inequality in~\eqref{eq:norm1h.ah} yields
    $$
    \sum_{n=1}^N\tau\norm[L^\infty(\Omega)]{c_h^n}^4
    \lesssim\sum_{n=1}^N\tau\norm[a,h]{\uchn}^2\norm[0,h]{\uLh\uchn}^2
    \lesssim\left(\max_{1\le n\le N}\norm[a,h]{\uchn}^{2}\right)\times
    \sum_{n=1}^N\tau\norm[0,h]{\uLh\uchn}^2.
    $$
    The first factor is $\lesssim 1$ owing to~\eqref{eq:a-priori:i}.
    Thus, to prove~\eqref{eq:a-priori:ii}, it suffices to show that also the second factor is $\lesssim 1$.
    Using the definition~\eqref{eq:uLh} of $\uLh$ followed by~\eqref{eq:discrete:2} with $\upsi=\uLh\uchn$, we infer that
    \begin{equation}
      \label{eq:a-priori:ii:intermediate}
      \gamma^2\norm[0,h]{\uLh\uchn}^2
      =-\gamma^2a_h(\uchn,\uLh\uchn)
      =(\Phi'(c_h^n),\Lh\uchn)-(w_h^n,\Lh\uchn).
    \end{equation}
    Using again~\eqref{eq:uLh} for the second term in the right-hand side of~\eqref{eq:a-priori:ii:intermediate} followed by the Cauchy--Schwarz and Young's inequalities, we obtain
    $$
    \begin{aligned}
      \gamma^2\norm[0,h]{\uLh\uchn}^2
      &=(\Phi'(c_h^n),\Lh\uchn)+a_h(\uchn,\uwhn)+s_{0,h}(\uLh\uchn,\uwhn)
      \\
      &\le
      \frac{1}{2\gamma^2}\norm{\Phi'(c_h^n)}^2
      {+} \frac{\gamma^2}{2}\norm[0,h]{\uLh\uchn}^2
      {+} \frac{\gamma^2}2\norm[a,h]{\uchn}^2
      {+} \frac1{2\gamma^2}\norm[a,h]{\uwhn}^2
      {+} \frac{1}{2\gamma^2}\seminorm[0,h]{\uwhn}^2.
    \end{aligned}
    $$
    Hence, since $\seminorm[0,h]{\uwhn}\le h\seminorm[1,h]{\uwhn}\lesssim\norm[a,h]{\uwhn}$,
    $$
    \gamma^2\norm[0,h]{\uLh\uchn}^2
    \lesssim
    \gamma^{-2}\norm{\Phi'(c_h^n)}^2
    + \gamma^2\norm[a,h]{\uchn}^2
    + \gamma^{-2}\norm[a,h]{\uwhn}^2.
    $$
    The fact that $\sum_{n=1}^n\tau\norm[0,h]{\uLh\uchn}^2\lesssim 1$ then follows multiplying the above inequality by $\tau$, summing over $1\le n\le N$, using~\eqref{eq:a-priori:i} to bound the second and third term in the right-hand side, and observing that
    \begin{equation}\label{eq:bnd.Phi'}
      \norm{\Phi'(c_h^n)}^2
      \le\norm[L^6(\Omega)]{c_h^n}^6 + 2\norm[L^4(\Omega)]{c_h^n}^4 + \norm{c_h^n}^2
      \lesssim \norm[1,h]{\uchn}^6 + \norm[1,h]{\uchn}^4 + \norm[1,h]{\uchn}^2
      \lesssim 1,
    \end{equation}
    where we have used the definition~\eqref{eq:Phi} to obtain the first bound, Friedrichs' inequality~\eqref{eq:friedrichs} with $r=6,4,2$ to obtain the second bound, and~\eqref{eq:a-priori:i} together with the first inequality in~\eqref{eq:norm1h.ah} to conclude.
    
  \item We proceed by proving that
    \begin{equation}\label{eq:a-priori:iii}
      \max_{1\le n\le N}\norm{w_h^n}^2
      + \gamma^2\sum_{n=1}^N\tau\norm{\dt c_h^n}^2
      \lesssim 1.
    \end{equation}
    Let $w_h^0\eqbydef\lproj{k+1}(\Phi'(c_h^0)-\gamma^2\LAPL c_0)$. Recalling~\eqref{eq:discrete.ic}, $w_h^0$ satisfies
    \begin{equation}\label{eq:uwhn0}
      (w_h^0,\psi_h) = (\Phi'(c_h^0),\psi_h) + \gamma^2 a_h(\uchn[0],\upsi)\qquad\forall\upsi\in\Uh.
    \end{equation}
    For any $1\le n\le N$, subtracting from~\eqref{eq:discrete:2} at time step $n$~\eqref{eq:discrete:2} at time step $(n-1)$ if $n>1$ or~\eqref{eq:uwhn0} if $n=1$, and selecting $\upsi=\uwhn$ as a test function in the resulting equation, it is inferred that
    $$
    (w_h^n-w_h^{n-1},w_h^n)=\tau\gamma^2a_h(\dt\uchn,\uwhn)+(\Phi'(c_h^n)-\Phi'(c_h^{n-1}),w_h^n).
    $$
    Using~\eqref{eq:discrete:1} with $\uphi=\tau\gamma^2\dt\uchn$ to infer $\tau\gamma^2a_h(\dt\uchn,\uwhn)=-\tau\gamma^2\norm{\dt c_h^n}^2$, we get
    \begin{equation}\label{eq:a-priori:4}
      (w_h^n-w_h^{n-1},w_h^n)
      + \tau\gamma^2\norm{\dt c_h^n}^2
      = (\Phi'(c_h^n)-\Phi'(c_h^{n-1}),w_h^n).
    \end{equation}
    From the fact that
    \begin{equation}\label{eq:magic.Phi'}
      \Phi'(r)-\Phi'(s)=(r^2+rs+s^2-1)(r-s),
    \end{equation}
    followed by the Cauchy--Schwarz and Young's inequalities, we infer
    \begin{equation}\label{eq:a-priori:5}
      |(\Phi'(c_h^n)-\Phi'(c_h^{n-1}),w_h^n)|
      \le
      \frac{\tau\gamma^2}{2}\norm{\dt c_h^n}^2 
      + \frac{\tau C^n}{2}\norm{w_h^n}^2,
    \end{equation}
    with $C^n\eqbydef C(1+\norm[L^\infty(\Omega)]{c_h^n}^4+\norm[L^\infty(\Omega)]{c_h^{n-1}}^4)$ for a real number $C>0$ independent of $h$ and $\tau$.
    Using~\eqref{eq:magic:BE} for the first term in the left-hand side of~\eqref{eq:a-priori:4} together with~\eqref{eq:a-priori:5} for the right-hand side, we get
    \begin{equation}\label{eq:a-priori:6}
      \norm{w_h^n}^2 + \norm{w_h^n-w_h^{n-1}}^2 + \tau\gamma^2\norm{\dt c_h^n}^2\le
      \tau C^n\norm{w_h^n}^2 + \norm{w_h^{n-1}}^2.
    \end{equation}
    Summing~\eqref{eq:a-priori:6} over $1\le n\le N$, observing that, thanks to~\eqref{eq:a-priori:ii} and the second bound in~\eqref{eq:bnd.uchn0} below, we can have $\tau C^n< 1$ for all $1\le n\le N$ provided that we choose $\tau$ small enough, and using the discrete Gronwall's inequality~\eqref{eq:disc.gronwall} (with $\delta=\tau$, $a^n=\norm{w_h^n}^2$, $b^n=\gamma^2\norm{\dt c_h^n}^2$, $\chi^n=C^n$ and $G=\norm{w_h^0}^2$), the estimate~\eqref{eq:a-priori:iii} follows if we can bound $\norm{w_h^0}^2$.
    To this end, recalling the definition of $w_h^0$ and using the Cauchy--Schwarz inequality, one has
    $$
    \norm{w_h^0}^2 = (\Phi'(c_h^0),w_h^0) - \gamma^2(\LAPL c_0,w_h^0)
    \le\left(\norm{\Phi'(c_h^0)} + \gamma^2\norm[H^2(\Omega)]{c_0}\right)\norm{w_h^0},
    $$
    and the conclusion follows from the regularity of $c_0$ noting the first bound in~\eqref{eq:bnd.uchn0} below and estimating the first term in parentheses as in~\eqref{eq:bnd.Phi'}.
    
  \item We conclude by proving the bound
    \begin{equation}\label{eq:a-priori:iv}
      \max_{1\le n\le N}\left(
      \norm[L^\infty(\Omega)]{c_h^n}
      + \norm[0,h]{\uLh\uchn}^2
      \right)
      \lesssim 1.
    \end{equation}
    Using the Cauchy--Schwarz and Young's inequalities to bound the right-hand side of~\eqref{eq:a-priori:ii:intermediate} followed by~\eqref{eq:friedrichs} with $r=6,4,2$ and the first inequality in~\eqref{eq:norm1h.ah}, we obtain, for all $1\le n\le N$,
    \begin{equation}\label{eq:a-priori:8}
      \begin{aligned}
        \gamma^2\norm[0,h]{\uLh\uchn}^2
        &\lesssim\gamma^{-2}\left(\norm{\Phi'(c_h^n)}^2 + \norm{w_h^n}^2\right)
        \\
        &\lesssim\left(
        \norm[L^6(\Omega)]{c_h^n}^6 + \norm[L^4(\Omega)]{c_h^n}^4 + \norm{c_h^n}^2 
        \right) + \norm{w_h^n}^2
        \\
        &\lesssim\left(
        \norm[a,h]{c_h^n}^6 + \norm[a,h]{c_h^n}^4 + \norm[a,h]{c_h^n}^2 
        \right) + \norm{w_h^n}^2 \lesssim 1,
      \end{aligned}
    \end{equation}
    where we have concluded using~\eqref{eq:a-priori:i} multiple times for the terms in parentheses and~\eqref{eq:a-priori:iii} for $\norm{w_h^n}^2$.
    Using the discrete Agmon's inequality~\eqref{eq:agmon} followed by Young's inequality and the first inequality in~\eqref{eq:norm1h.ah}, we infer
    $$
    \max_{1\le n\le N}\norm[L^\infty(\Omega)]{c_h^n}\lesssim 
    \max_{1\le n\le N}\left( \norm[a,h]{\uchn} + \norm[0,h]{\uLh\uchn} \right)    
    \lesssim 1,
    $$
    where the conclusion follows using~\eqref{eq:a-priori:i} for the first addend in the argument of the maximum and~\eqref{eq:a-priori:8} for the second.
  \end{asparaenum}
\end{proof}

\begin{proposition}[{Bounds for $\uchn[0]$}]
  Let $\uchn[0]\in\UhO$ be defined by~\eqref{eq:discrete.ic} from an initial datum $c_0\in H^2(\Omega)\cap L^2_0(\Omega)$ such that $\partial_{\normal} c_0=0$ on $\partial\Omega$.
  It holds, with real number $C>0$ independent of $h$,
  \begin{equation}\label{eq:bnd.uchn0}
    \norm[1,h]{\uchn[0]} + \norm[L^\infty(\Omega)]{c_h^0}\le C\norm[H^2(\Omega)]{c_0}.
  \end{equation}
\end{proposition}

\begin{proof}
  To prove the first bound in~\eqref{eq:bnd.uchn0}, let $\uphi=\uchn[0]$ in~\eqref{eq:discrete.ic} and use the first inequality in~\eqref{eq:norm1h.ah}, the Cauchy--Schwarz inequality and the discrete Poincar\'e's inequality~\eqref{eq:friedrichs} with $r=2$ to infer
  $$
  \norm[1,h]{\uchn[0]}^2\lesssim a_h(\uchn[0],\uchn[0])=-(\LAPL c_0,c_h^0)\le\norm{\LAPL c_0}\norm{c_h^0}\lesssim\norm[H^2(\Omega)]{c_0}\norm[1,h]{\uchn[0]}.
  $$
  To prove the second bound in~\eqref{eq:bnd.uchn0}, we start by noticing that, using the definition~\eqref{eq:uLh} of $\uLh$ with $\uzh=-\uLh\uchn[0]$,
  $$
  \norm[0,h]{\uLh\uchn[0]}^2 = -a_h(\uchn[0],\uLh\uchn[0])
  = (\LAPL c_0,\Lh\uchn[0])
  \le\norm[H^2(\Omega)]{c_0}\norm{\Lh\uchn[0]},
  $$
  hence $\norm[0,h]{\uLh\uchn[0]}\le\norm[H^2(\Omega)]{c_0}$.
  Combining the discrete Agmon's inequality~\eqref{eq:agmon} with the latter inequality and the first bound in~\eqref{eq:bnd.uchn0}, one gets
  $$
  \norm[L^\infty(\Omega)]{c_h^0}\le\norm[1,h]{\uchn[0]}^{\frac12}\norm[0,h]{\uLh\uchn[0]}^{\frac12}\lesssim\norm[H^2(\Omega)]{c_0},
  $$
  and the desired result follows.
\end{proof}

%------------------------------------------------------------------------------%

\section{Error analysis}\label{sec:err.anal}

In this section we carry out the error analysis of the method~\eqref{eq:discrete}.

\subsection{Error equations}\label{sec:err.anal:err.eq}

Our goal is to estimate the difference between the discrete solution obtained solving~\eqref{eq:discrete} and the projections of the exact solution such that, for all $1\le n\le N$, $\uhwhn=\Ih w^n$, while, for all $0\le n\le N$, $\uhchn\in\UhO$ solves
\begin{equation*}\label{eq:uhchn}
  a_h(\uhchn,\uphi) = -(\LAPL c^n,\varphi_h)\qquad\forall\uphi\in\Uh,
\end{equation*}
and $(\hchn,1)=0$.
We define, for all $1\le n\le N$, the errors
\begin{equation}\label{eq:errors}
  \uech\eqbydef\uchn-\uhchn,\qquad
  \uewh\eqbydef\uwhn-\uhwhn.
\end{equation}
By definition~\eqref{eq:discrete.ic}, $\uhchn[0]=\uchn[0]$, which prompts us to set $\uech[0]\eqbydef\underline{0}$.
Using Poincar\'{e}'s inequality~\eqref{eq:friedrichs} with $r=2$ and the consistency~\eqref{eq:cons.ah} of $a_h$, the following estimate is readily inferred:
For all $0\le n\le N$, assuming the additional regularity $c^n\in H^{k+2}(\Omega)$,
\begin{equation}\label{eq:approx.uhchn}
  \norm{\hchn-\lproj{k+1}c^n}
  \lesssim \norm[1,h]{\uhchn-\Ih c^n}
  \lesssim h^{k+1}\norm[H^{k+2}(\Omega)]{c^n}.
\end{equation}
\begin{remark}[Improved $L^2$-estimate]
  We notice, in passing, that, using elliptic regularity (which holds since $\Omega$ is convex, cf., e.g.,~\cite{Grisvard:92}), one can improve this result and show that $\norm{\hchn-\lproj{k+1}c^n}\lesssim h^{k+2}\norm[H^{k+2}(\Omega)]{c^n}$.
\end{remark}
Recalling~\eqref{eq:discrete}, for all $1\le n\le N$, the error $(\uech,\uewh)\in\UhO\times\Uh$ solves
\begin{subequations}\label{eq:err.eq}
  \begin{align}
    \label{eq:err.eq:1}
    &(\dt\ech,\varphi_h) + a_h(\uewh,\uphi) = \mathcal{E}(\uphi)
    &\qquad&\forall\uphi\in\Uh,
    \\
    \label{eq:err.eq:2}
    &(\ewh,\psi_h) = (\Phi'(c_h^n) - \Phi'(c^n),\psi_h) + \gamma^2 a_h(\uech,\upsi),
    &\qquad&\forall\upsi\in\Uh,
  \end{align}
\end{subequations}
where, in~\eqref{eq:err.eq:1}, we have defined the consistency error
\begin{equation}
  \label{eq:calE}
  \mathcal{E}(\uphi)\eqbydef -(\dt\hchn,\varphi_h) - a_h(\uhwhn,\uphi),
\end{equation}
while in~\eqref{eq:err.eq:2} we have combined the definitions of $\uhwhn$ and $\uhchn$ with~\eqref{eq:strong:2} to infer
$$
(\hwhn,\psi_h)- \gamma^2 a_h(\uhchn,\upsi)
= (w^n + \LAPL c^n,\psi_h)
= (\Phi'(c^n),\psi_h).
$$

\subsection{Error estimate}

\begin{theorem}[Error estimate]\label{thm:err.est}
  Suppose that the assumptions of Lemma~\ref{lem:a-priori} hold true.
  Let $(c,w)$ denote the solution to~\eqref{eq:strong}, for which we assume the following additional regularity:
  \begin{equation}\label{eq:reg}
    c\in C^2([0,\tF];L^2(\Omega))\cap C^1([0,\tF];H^{k+2}(\Omega)),\qquad
    w\in C^0([0,\tF];H^{k+2}(\Omega)).
  \end{equation}
  Then, the following estimate holds for the errors defined by~\eqref{eq:errors}:
  \begin{equation}\label{eq:err.est}
    \left(
    \max_{1\le n\le N}\norm[a,h]{\uech}^2 + \sum_{n=1}^N\tau\norm[a,h]{\uewh}^2
    \right)^{\frac12}
    \le C (h^{k+1} + \tau),
  \end{equation}
  with real number $C>0$ independent of $h$ and $\tau$.
\end{theorem}

\begin{proof}
  Let $1\le n\le N$.
  Subtracting~\eqref{eq:err.eq:2} with $\upsi=\dt\uech$ from~\eqref{eq:err.eq:1} with $\uphi=\uewh$, we obtain
  \begin{equation}\label{eq:err.eq:basic}
    \norm[a,h]{\uewh}^2 + \gamma^2 a_h(\uech,\dt\uech)
    = \mathcal{E}(\uewh) + (\Phi'(c^n)-\Phi'(c_h^n),\dt\ech)
    \eqbydef\term_1 + \term_2.
  \end{equation}
  We proceed to bound the terms in the right-hand side.

  \begin{asparaenum}[(i)]
  \item \emph{Bound for $\term_1$.}
    Let $\uphi\in\Uh$.
    Adding to~\eqref{eq:calE} the quantity $$(\ud_t c^n-\LAPL w^n,\varphi_h)+(\dt\lproj{k+1}c^n-\dt c^n,\varphi_h)=0,$$ (use~\eqref{eq:strong:1} to prove that the first addend is 0 and the definition of the $L^2$-orthogonal projector $\lproj{k+1}$ to prove that the second is also 0), we can decompose $\mathcal{E}(\uphi)$ as follows:
    $$
    \begin{aligned}
      \mathcal{E}(\uphi)
      &= (\ud_t c^n - \dt c^n,\varphi_h)
      +(\dt(\lproj{k+1} c^n-\hchn),\varphi_h)
      - \left( a_h(\uhwhn,\uphi) + (\LAPL w^n,\varphi_h) \right)
      \\
      &\eqbydef\term_{1,1}+\term_{1,2}+\term_{1,3}.
    \end{aligned}
    $$
    For the first term, we have
    \begin{equation}\label{eq:err.est:T1:T11}
      |\term_{1,1}|
      \le\norm{\ud_t c^n - \dt c^n}\norm{\varphi_h}    
      \lesssim\tau\norm[{C^2([0,\tF];L^2(\Omega))}]{c}\norm[1,h]{\uphi}
      \lesssim\tau\norm[1,h]{\uphi},
    \end{equation}
    where we have used the Cauchy--Schwarz inequality, a classical estimate based on Taylor's remainder, Poincar\'e's inequality~\eqref{eq:friedrichs} with $r=2$, and we have concluded using the regularity~\eqref{eq:reg} for $c$.
    For the second term, on the other hand, using the Cauchy--Schwarz inequality followed by~\eqref{eq:approx.uhchn} together with the $C^1$-stability of the backward differencing operator~\eqref{eq:ddt}, Poincar\'e's inequality, and the regularity~\eqref{eq:reg} for $c$, we readily obtain
    \begin{equation}\label{eq:err.est:T1:T12}
      |\term_{1,2}|
      \le\norm{\dt(\lproj{k+1} c^n-\hchn)}\norm{\varphi_h}
      \lesssim h^{k+1}\norm[{C^1([0,\tF];H^{k+2}(\Omega))}]{c^n}\norm{\varphi_h}
      \lesssim h^{k+1}\norm[1,h]{\uphi}.
    \end{equation}
    Finally, recalling the consistency properties~\eqref{eq:cons.ah} of $a_h$, we get for the last term
    \begin{equation}\label{eq:err.est:T1:T13}
      \begin{aligned}
        |\term_{1,3}|
        &\lesssim h^{k+1}\norm[H^{k+2}(\Omega)]{w^n}\norm[1,h]{\uphi}
        \le h^{k+1}\norm[{C^0([0,\tF];H^{k+2}(\Omega))}]{w}\norm[1,h]{\uphi}
        \\
        &\lesssim h^{k+1}\norm[1,h]{\uphi}.
      \end{aligned}
    \end{equation}
    Collecting the bounds~\eqref{eq:err.est:T1:T11}--\eqref{eq:err.est:T1:T13}, it is inferred that
    \begin{equation}\label{eq:est.$}
      \$\eqbydef\sup_{\uphi\in\Uh,\norm[1,h]{\uphi}=1}\mathcal{E}(\uphi)\lesssim h^{k+1} + \tau,
    \end{equation}
    so that, for any real $\epsilon>0$, denoting by $C_\epsilon>0$ a real depending on $\epsilon$ but not on $h$ or $\tau$, and using the second inequality in~\eqref{eq:norm1h.ah} to bound $\norm[1,h]{\uewh}\lesssim\norm[a,h]{\uewh}$,
    \begin{equation}\label{eq:err.est:T1}
      |\term_1|
      \le\$\norm[1,h]{\uewh}
      \lesssim (h^{k+1} + \tau)\norm[1,h]{\uewh}
      \le C_\epsilon(h^{k+1} + \tau)^2 + \epsilon\norm[a,h]{\uewh}^2.
    \end{equation}
    
  \item \emph{Bound for $\term_2$.}
    Set, for the sake of brevity, $Q^n\eqbydef\Phi'(c_h^n) - \Phi'(c^n)$, and define the DOF vector $\uzh\in\Uh$ such that
    \begin{equation}\label{eq:def.zh}
      z_T=\lproj[T]{k+1}Q^n\quad\forall T\in\Th,\qquad
      z_F=\lproj[F]{k}\avg{Q^n}\quad\forall F\in\Fhi,\qquad
      z_F=\lproj[F]{k} z_{T_F}\quad\forall F\in\Fhb
    \end{equation}
    where $\avg{{\cdot}}$ denotes the usual average operator such that, for any function $\varphi$ admitting a possibly two-valued trace on $F\in\Fh[T_1]\cap\Fh[T_2]$, $\avg{\varphi}\eqbydef\frac12(\restrto{\varphi}{T_1}+\restrto{\varphi}{T_2})$, while, for a boundary face $F\in\Fhb$, $T_F$ denotes the unique element in $\Th$ such that $F\in\Fh[T_F]$.
    We have, using the definition of $\lproj[T]{k+1}$ followed by~\eqref{eq:err.eq:1} with $\uphi=\uzh$,~\eqref{eq:est.$}, and the second inequality in~\eqref{eq:norm1h.ah},
    \begin{equation}\label{eq:term2.basic}
      \term_2 = (z_h,\dt\ech)
      = \mathcal{E}(\uzh) - a_h(\uewh,\uzh)
      \lesssim\left(\$ + \norm[a,h]{\uewh}\right)\norm[1,h]{\uzh}.
    \end{equation}
    By Proposition~\ref{prop:bnd.norm1hzh} below,
    \begin{equation}\label{eq:bnd.norm1hzh}
      \norm[1,h]{\uzh}\lesssim\norm[a,h]{\uech} + h^{k+1},
    \end{equation}
    hence, for any real $\epsilon>0$, denoting by $C_\epsilon>0$ a real number depending on $\epsilon$ but not on $h$ or $\tau$, and recalling the bound~\eqref{eq:est.$} for $\$$,
    \begin{equation}\label{eq:err.est:T2}
      |\term_2|\le C_\epsilon\left(\norm[a,h]{\uech}^2 + (h^{k+1} + \tau)^2\right) + \epsilon\norm[a,h]{\uewh}^2.
    \end{equation}
    
  \item \emph{Conclusion.}
    Using~\eqref{eq:err.est:T1} and~\eqref{eq:err.est:T2} with $\epsilon=\frac14$ to bound the right-hand side of~\eqref{eq:err.eq:basic}, it is inferred
    $$
    \norm[a,h]{\uewh}^2 + \gamma^2 a_h(\uech,\dt\uech)
    \lesssim (h^{k+1} + \tau)^2 + \norm[a,h]{\uech}^2.
    $$
    Multiplying by $\tau$, summing over $1\le n\le N$, using~\eqref{eq:magic:BE} for the second term in the left-hand side, and recalling that, by definition, $\uech[0]=\underline{0}$, we get
    $$
    \gamma^2\norm[a,h]{\uech[N]}^2 + \sum_{n=1}^N\tau\norm[a,h]{\uewh}^2\le \sum_{n=1}^NC\tau\norm[a,h]{\uech}^2 + C(h^{k+1} + \tau)^2,
    $$
    with $C>0$ independent of $h$ and $\tau$.
    The error estimate~\eqref{eq:err.est} then follows from an application of the discrete Gronwall's inequality~\eqref{eq:disc.gronwall} with $\delta=\tau$, $a^n=\gamma^2\norm[a,h]{\uech}^2$, $b^n=\norm[a,h]{\uewh}^2$, $\chi^n=C$, and $G=C(h^{k+1}+\tau)^2$ assuming $\tau$ small enough.
  \end{asparaenum}
\end{proof}

\begin{remark}[BDF2 time discretization]\label{rem:bdf2}
  In Section~\ref{sec:num.tests}, we have also used a BDF2 scheme to march in time, which corresponds to the backward differencing operator
  $$
  \bdf\varphi\eqbydef\frac{3\varphi^{n+2}-4\varphi^{n+1}+\varphi^n}{2\tau},
  $$
  used in place of~\eqref{eq:ddt}.
  The analysis is essentially analogous to the backward Euler scheme, the main difference being that formula~\eqref{eq:magic:BE} is replaced by
  $$
  2x (3x-4y+z) = x^{2} - y^{2} + (2x-y)^{2} - (2y-z)^{2} + (x-2y+z)^{2}.
  $$
  As a result, the right-hand side of~\eqref{eq:err.est} scales as $(h^{k+1}+\tau^2)$ instead of $(h^{k+1}+\tau)$.
\end{remark}

To prove the bound~\eqref{eq:bnd.norm1hzh}, we need discrete counterparts of the following Gagliardo--Nirenberg--Poincar\'{e}'s inequalities valid for $p\in[2,+\infty)$ if $d=2$, $p\in[2,6]$ if $d=3$, and all \mbox{$v\in H^2(\Omega)\cap L^2_0(\Omega)$}:
\begin{equation}\label{eq:gnp.cont}
  \seminorm[W^{1,p}(\Omega)]{v}\lesssim\norm{v}^{1-\alpha}\seminorm[H^2(\Omega)]{v}^{\alpha}
  \lesssim\seminorm[H^1(\Omega)]{v}^{1-\alpha}\seminorm[H^2(\Omega)]{v}^{\alpha},
  \qquad\alpha\eqbydef\frac12 + \frac{d}2\left(\frac12-\frac1p\right),
\end{equation}
where the first bound follows from~\cite[Theorem~3]{Adams.Fournier:77} and the second from Poincar\'{e}'s inequality.
The proof of the following Lemma will be given in Appendix~\ref{sec:proofs}.
\begin{lemma}[Discrete Gagliardo--Nirenberg--Poincar\'{e}'s inequalities]\label{lem:gnp}
  Under the assumptions of Lemma~\ref{lem:agmon}, it holds for $p\in [2,+\infty)$ if $d=2$, $p\in[2,6]$ if $d=3$ with $C>0$ independent of $h$ and $\alpha$ defined as in \eqref{eq:gnp.cont},
  \begin{equation}\label{eq:gnp}
    \forall \uvh\in\UhO,\qquad    
    \norm[L^p(\Omega)^d]{\GRADh v_h}\le C\norm[1,h]{\uvh}^{1-\alpha}\norm[0,h]{\uLh\uvh}^{\alpha}.
  \end{equation}  
\end{lemma}

\begin{proposition}[{Bound on $\norm[1,h]{\uzh}$}]\label{prop:bnd.norm1hzh}
  With $\uzh$ defined as in~\eqref{eq:def.zh}, the bound~\eqref{eq:bnd.norm1hzh} holds.
\end{proposition}

\begin{proof}
  Recalling the definition~\eqref{eq:norm1h} of the $\norm[1,h]{{\cdot}}$-norm, one has
  \begin{equation}\label{eq:bnd.norm1hzh:basic}
    \norm[1,h]{\uzh}^2 = \norm{\GRADh\lproj{k+1} Q^n}^2
    + \sum_{T\in\Th}\sum_{F\in\Fh[T]\cap\Fhi}h_F^{-1}\norm[F]{\lproj[F]{k}(\avg{Q^n} - \lproj[T]{k+1}Q^n)}^2\eqbydef\term_1^2 + \term_2^2.
  \end{equation}
  \begin{asparaenum}[(i)]
  \item \emph{Bound for $\term_1$.}
    Using the $H^1$-stability~\eqref{eq:lproj.stab} of $\lproj{k+1}$, formula~\eqref{eq:magic.Phi'} to infer $Q^n=q^n(c_h^n-c^n)$ with $q^n\eqbydef (c_h^n)^2 + c_h^nc^n + (c^n)^2 - 1$, the triangle and H\"{o}lder inequalities, we get, for all $T\in\Th$,
    $$
    \begin{aligned}
      |\term_1|
      &\lesssim\norm{\GRADh Q^n}
      \le\norm{q^n\GRADh(c_h^n - c^n)} + \norm{(c_h^n-c^n)\GRADh q^n}
      \\
      &\lesssim
      \left(
      \norm[L^\infty(\Omega)]{c_h^n}^2 + \norm[L^\infty(\Omega)]{c^n}^2 + 1
      \right)\norm{\GRADh(c_h^n-c^n)}
      \\
      &\quad
      + \norm[L^6(\Omega)]{c_h^n-c^n}\left(\norm[L^\infty(\Omega)]{c_h^n} + \norm[L^\infty(\Omega)]{c^n}\right)
      \left(\norm[L^3(\Omega)^d]{\GRADh c_h^n} + \norm[L^3(\Omega)^d]{\GRAD c^n}\right).
    \end{aligned}
    $$
    Noting the a priori bound~\eqref{eq:a-priori:iv} and the regularity assumption~\eqref{eq:reg}, both $\norm[L^\infty(\Omega)]{c_h^n}$ and $\norm[L^\infty(\Omega)]{c^n}$ are $\lesssim 1$.
    Additionally, by the continuous Gagliardo--Nirenberg--Poincar\'{e}'s inequality~\eqref{eq:gnp.cont} with $p=3$ and the regularity assumption~\eqref{eq:reg}, one has with $\alpha=\nicefrac12+\nicefrac{d}{12}$,
    $\norm[L^3(\Omega)^d]{\GRAD c^n}\lesssim\seminorm[H^1(\Omega)]{c^n}^{1-\alpha}\norm[H^2(\Omega)]{c^n}^{\alpha}\lesssim 1$.
    Similarly, the discrete Gagliardo--Nirenberg--Poincar\'{e}'s inequality~\eqref{eq:gnp} with $p=3$ combined with the a priori bounds~\eqref{eq:a-priori:i} and~\eqref{eq:a-priori:iv} yields $\norm[L^3(\Omega)^d]{\GRAD_h c_h^n}\lesssim\norm[1,h]{\uchn}^{1-\alpha}\norm[0,h]{\uLh\uchn}^{\alpha}\lesssim 1$.
    Then, inserting $\pm(\hchn-\lproj{k+1}c^n)$ and using the triangle inequality,
    \begin{equation}\label{eq:need.k+1}
      \begin{aligned}
        |\term_1|
        &\lesssim\left(\norm{\GRADh\ech}+\norm[L^6(\Omega)]{\ech}\right)
        +\left(\norm{\GRADh(\hchn-\lproj{k+1}c^n)}+\norm[L^6(\Omega)]{\hchn-\lproj{k+1}c^n}\right)
        \\
        &\quad +\left(\norm{\GRADh(\lproj{k+1}c^n-c^n)}+\norm[L^6(\Omega)]{\lproj{k+1}c^n-c^n}\right)
        \eqbydef\term_{1,1} + \term_{1,2} + \term_{1,3}.
      \end{aligned}
    \end{equation}
    Using the discrete Friedrichs' inequality~\eqref{eq:friedrichs} with $r=6$ together with the definition~\eqref{eq:norm1h} of the $\norm[1,h]{{\cdot}}$-norm and the first inequality in~\eqref{eq:norm1h.ah}, it is readily inferred that $\term_{1,1}\lesssim\norm[a,h]{\uech}$.
    Again the Friedrichs' inequality~\eqref{eq:friedrichs} with $r=6$ followed by the approximation properties~\eqref{eq:approx.uhchn} of $\uhchn$ and the regularity~\eqref{eq:reg} yields $\term_{2,2}\lesssim h^{k+1}\norm[H^{k+2}(\Omega)]{c^n}\lesssim h^{k+1}$.
    Finally, using the approximation properties~\eqref{eq:lproj.approx} of $\lproj{k+1}$, we have $\term_{1,3}\lesssim h^{k+1}(\norm[H^{k+2}(\Omega)]{c^n} + \norm[W^{k+1,6}(\Omega)]{c^n})\lesssim h^{k+1}$, where we have used the fact that $H^{k+2}(\Omega)\subset W^{k+1,6}(\Omega)$ for all $k\ge 0$ and $d\in\{2,3\}$ on domains satisfying the cone condition  (cf.~\cite[Theorem~4.12]{Adams:03}).
    Gathering the previous bounds, we conclude that
    \begin{equation}\label{eq:bnd.norm1hzh:T1}
      |\term_1|\lesssim\norm[a,h]{\uech} + h^{k+1}.
    \end{equation}
    
  \item \emph{Bound for $\term_2$.} For all interface $F\in\Fh[T_1]\cap\Fh[T_2]$, we denote by $\jump{{\cdot}}$ the usual jump operator such that, for every function $\varphi$ with a possibly two-valued trace on $F$, $\jump{\varphi}\eqbydef\restrto{\varphi}{T_1}-\restrto{\varphi}{T_2}$ (the orientation is irrelevant).
    Let an element $T\in\Th$ and an interface face $F\in\Fh[T]\cap\Fh[T^+]$ be fixed.
    Using the $L^2$-stability of $\lproj[F]{k}$, inserting $\pm Q_T^n$ (with $Q_T^n\eqbydef\restrto{Q^n}{T}$), and using the triangle inequality it holds,
    \begin{equation}\label{eq:bnd.norm1hzh:1}
      \begin{aligned}
        \norm[F]{\lproj[F]{k}(\avg{Q^n} - \lproj[T]{k+1} Q_T^n)}
        &\le \norm[F]{\avg{Q^n} - \lproj[T]{k+1} Q_T^n}
        \\
        &\le \frac12\norm[F]{\jump{Q^n}} + \norm[F]{Q_T^n-\lproj[T]{k+1}Q_T^n}
        \\
        &\lesssim \norm[F]{\jump{Q^n}} + h_T^{\frac12}\norm[T]{\GRAD Q_T^n},
      \end{aligned}
    \end{equation}
    where we have used~\eqref{eq:lproj.approx} for the last term.
    Let us bound the first term in the right-hand side.
    Observing that $\jump{\Phi'(c^n)}=0$ and recalling~\eqref{eq:magic.Phi'}, it is inferred
    $$
    |\jump{Q^n}|
    =|\jump{\Phi'(c_h^n)}|
    \le|\jump{c_h^n}|\left(|c_{T}|^2+|c_{T}||c_{T^+}|+|c_{T^+}|^2+1\right).
    $$
    Using this relation, and noticing the a priori bound~\eqref{eq:a-priori:iv}, we get
    $$
    \norm[F]{\jump{Q^n}}\lesssim\left(\norm[L^\infty(\Omega)]{c_h^n}^2+1\right)\norm[F]{\jump{c_h^n}}
    \lesssim\norm[F]{\jump{c_h^n}} = \norm[F]{\jump{c_h^n-c^n}},
    $$
    where the conclusion follows observing that $c^n$ has zero jumps across interfaces.
    Inserting $\pm\jump{\hchn-\lproj{k+1}c^n}$ inside the norm and using the triangle inequality, we obtain
    \begin{equation}\label{eq:bnd.jumpQn}
      \norm[F]{\jump{Q^n}}
      \lesssim\norm[F]{\jump{c_h^n-\hchn}} + \norm[F]{\jump{\hchn-\lproj{k+1}c^n}} + \norm[F]{\jump{\lproj{k+1} c^n - c^n}}.
    \end{equation}
    
    Define on $H^1(\Th)$ the jump seminorm $\seminorm[\rm J]{v}^2\eqbydef\sum_{F\in\Fhi}h_F^{-1}\norm[F]{\jump{v}}^2$.
    Let us prove that
    \begin{equation}\label{eq:bnd.seminormJ}
      \forall\uvh\in\Uh,\qquad
      \seminorm[\rm J]{v_h}\lesssim\norm[1,h]{\uvh}\lesssim\norm[a,h]{\uvh}.
    \end{equation}
    Inserting $\pm(\lproj[F]{k}\jump{v_h}-v_F)$ and using the triangle inequality, it is inferred that
    $$
    \seminorm[\rm J]{v_h}^2
    \lesssim\sum_{F\in\Fhi}\sum_{T\in\Th[F]} h_F^{-1}\left(
    \norm[F]{v_T-\lproj[F]{k}v_T}^2 + \norm[F]{\lproj[F]{k}(v_T-v_F)}^2
    \right)
    \lesssim\norm{\GRADh v_h}^2 + \seminorm[1,h]{\uvh}^2,
    $$
    where we have used~\eqref{eq:lprojF.approx} followed by the discrete trace inequality~\eqref{eq:trace.disc} and the fact that $\card{\Fh[T]}\lesssim 1$ by mesh regularity for the first term, and the definition~\eqref{eq:norm1h} of the $\seminorm[1,h]{{\cdot}}$-seminorm for the second term.
    This proves the first bound in~\eqref{eq:bnd.seminormJ}. 
    The second bound follows from~\eqref{eq:norm1h.ah}.
    
    Multiplying~\eqref{eq:bnd.norm1hzh:1} by $h_F^{-\frac12}$, squaring, summing over $F\in\Fh[T]\cap\Fhi$ then over $T\in\Th$, using mesh regularity to infer that $\card{\Fh[T]}$ is bounded uniformly in $h$, and noticing~\eqref{eq:bnd.jumpQn} yields
    \begin{equation}\label{eq:need.k+1'}
      \begin{aligned}
        \term_2^2
        &\lesssim\norm{\GRADh Q^n}^2
        + \seminorm[\rm J]{c_h^n-\hchn}^2
        + \seminorm[\rm J]{\hchn-\lproj{k+1}c^n}^2
        + \seminorm[\rm J]{\lproj{k+1} c^n - c^n}^2
        \\
        &\lesssim\norm{\GRADh Q^n}^2 + \norm[a,h]{\uech}^2 + \norm[a,h]{\uhchn-\Ih c^n}^2 + \seminorm[\rm J]{\lproj{k+1} c^n - c^n}^2
        \\
        &\lesssim\norm{\GRADh Q^n}^2 + \norm[a,h]{\uech}^2 + \left(h^{k+1}\norm[H^{k+2}(\Omega)]{c^n}\right)^2,
      \end{aligned}
    \end{equation}
    where we have used~\eqref{eq:bnd.seminormJ} to pass to the second line and the approximation properties~\eqref{eq:approx.uhchn} of $\uhchn$ and~\eqref{eq:lproj.approx} of $\lproj{k+1}$ to conclude.
    Proceeding as in point (i) to bound the first term in the right-hand side of~\eqref{eq:need.k+1'}, and recalling the regularity assumptions~\eqref{eq:reg} on $c$, we conclude
    \begin{equation}\label{eq:bnd.norm1hzh:T2}
      |\term_2|\le\norm[a,h]{\uech} + h^{k+1}.
    \end{equation}
    
  \item \emph{Conclusion.} Using~\eqref{eq:bnd.norm1hzh:T1} and~\eqref{eq:bnd.norm1hzh:T2} in~\eqref{eq:bnd.norm1hzh:basic}, the estimate~\eqref{eq:bnd.norm1hzh} follows.
  \end{asparaenum}
\end{proof}

\begin{remark}[Polynomial degree for element DOFs]\label{rem:need.k+1}
  The use of polynomials of degree $(k+1)$ (instead of $k$) as elements DOFs in the discrete space~\eqref{eq:Uh} is required to infer an estimate of order $h^{k+1}$ in~\eqref{eq:need.k+1} and for the last term in~\eqref{eq:need.k+1'}.
\end{remark}

%------------------------------------------------------------------------------%

\section{Numerical results}\label{sec:num.tests}
%% Print convergence plots in black and white
\pgfqkeys{/pgfplots}{ cycle list name = black white }
In this section we provide numerical evidence to confirm the theoretical results. %% and compare some results with the discontinuous Galerkin (dG) method.

\subsection{Convergence}

\begin{figure}
  \centering
  \includegraphics[height=3.5cm]{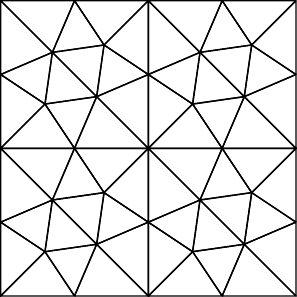}
  \hspace{0.25cm}
  \includegraphics[height=3.5cm]{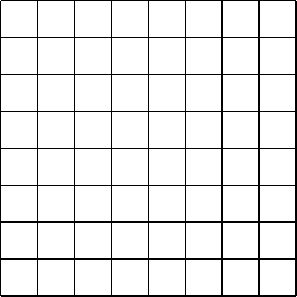}
  \hspace{0.25cm}
  \includegraphics[height=3.5cm]{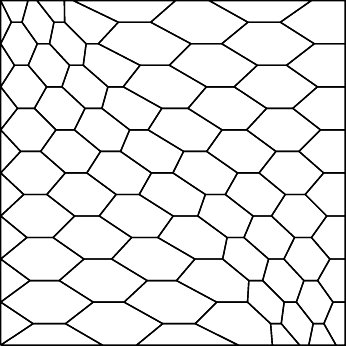}  
  \caption{Mesh families for the numerical tests\label{mesh_families}}
\end{figure}

We start by a non-physical numerical test that demonstrates the orders of convergence achieved by our method.
We solve the Cahn-Hilliard problem~\eqref{eq:err.eq} on the unit square $\Omega = (0,1)^2$ with $t_F = 1$, order-parameter
\begin{equation}
  \label{eq:cis_cv}
  c(\vec{x},t) = t \cos(\pi x_1) \cos(\pi x_2),
\end{equation}
and chemical potential $w$ inferred from $c$ according to~\eqref{eq:strong:2}.
The right-hand side of~\eqref{eq:strong:1} is also modified by introducing a nonzero source in accordance with the expression of $c$.
The interface parameter $\gamma$ is taken equal to 1.

We consider the triangular, Cartesian, and (predominantly) hexagonal mesh families of Figure~\ref{mesh_families}.
The two former mesh families were introduced in the FVCA5 benchmark~\cite{Herbin.Hubert:08}, whereas the latter was introduced in~\cite{Di-Pietro.Lemaire:15}.
To march in time, we use the implicit Euler scheme.
Since the order-parameter is linear in time, only the spatial component of the discretization error is nonzero and the choice of the time step is irrelevant.
The energy errors $\norm[a,h]{\uchn[N]-\Ih c^N}$ and $\norm[a,h]{\uwhn[N] - \Ih w^N}$ at final time are depicted in Figure~\ref{tab_cv:en}.
For all mesh families, the convergence rate is $(k+1)$, in accordance with Theorem~\ref{thm:err.est}.
For the sake of completeness, we also display in Figure~\ref{tab_cv:L2} the $L^2$-errors $\norm{c_h^n - \lproj{k+1} c^n}$ and $\norm{w_h^n - \lproj{k+1} w^n}$, for which an optimal convergence rate of $(k+2)$ is observed.

%------------------------------------------------------------------------------%
% Energy errors
%------------------------------------------------------------------------------%

\begin{figure}
  \begin{minipage}[b]{\textwidth}
    %% Triangular, err_ch_en
    \begin{tikzpicture}[scale=0.50]
      \begin{loglogaxis}[
          xmin = 0.001,
          legend style = {
            legend pos = south west 
          }
        ]
        \addplot table[x=meshsize,y={create col/linear regression={y=err_ch_en}}] {ch_0_mesh1.dat}
        coordinate [pos=0.75] (A)
        coordinate [pos=1.00] (B);
        \xdef\slopea{\pgfplotstableregressiona}
        \draw (A) -| (B) node[pos=0.75,anchor=east] {\pgfmathprintnumber{\slopea}};
        \addplot table[x=meshsize,y={create col/linear regression={y=err_ch_en}}] {ch_1_mesh1.dat}
        coordinate [pos=0.75] (A)
        coordinate [pos=1.00] (B);
        \xdef\slopeb{\pgfplotstableregressiona}
        \draw (A) -| (B) node[pos=0.75,anchor=south east] {\pgfmathprintnumber{\slopeb}};
        \addplot table[x=meshsize,y={create col/linear regression={y=err_ch_en}}] {ch_2_mesh1.dat}
        coordinate [pos=0.75] (A)
        coordinate [pos=1.00] (B);
        \xdef\slopec{\pgfplotstableregressiona}
        \draw (A) -| (B) node[pos=0.75,anchor=east] {\pgfmathprintnumber{\slopec}};
        \legend{$k=0$,$k=1$,$k=2$};
      \end{loglogaxis}
    \end{tikzpicture}
    \hfill
    %% Cartesian, err_ch_en
    \begin{tikzpicture}[scale=0.50]
      \begin{loglogaxis}[
          xmin = 0.002,
          legend style = {
            legend pos = south west
          }
        ]
        \addplot table[x=meshsize,y={create col/linear regression={y=err_ch_en}}] {ch_0_mesh2.dat}
        coordinate [pos=0.75] (A)
        coordinate [pos=1.00] (B);
        \xdef\slopea{\pgfplotstableregressiona}
        \draw (A) -| (B) node[pos=0.75,anchor=east] {\pgfmathprintnumber{\slopea}};
        \addplot table[x=meshsize,y={create col/linear regression={y=err_ch_en}}] {ch_1_mesh2.dat}
        coordinate [pos=0.75] (A)
        coordinate [pos=1.00] (B);
        \xdef\slopeb{\pgfplotstableregressiona}
        \draw (A) -| (B) node[pos=0.75,anchor=south east] {\pgfmathprintnumber{\slopeb}};
        \addplot table[x=meshsize,y={create col/linear regression={y=err_ch_en}}] {ch_2_mesh2.dat}
        coordinate [pos=0.75] (A)
        coordinate [pos=1.00] (B);
        \xdef\slopec{\pgfplotstableregressiona}
        \draw (A) -| (B) node[pos=0.75,anchor=east] {\pgfmathprintnumber{\slopec}};
        \legend{$k=0$,$k=1$,$k=2$};
      \end{loglogaxis}
    \end{tikzpicture}
    \hfill
    %% Hexagonal, err_ch_en
    \begin{tikzpicture}[scale=0.50]
      \begin{loglogaxis}[
          xmin = 0.0025,
          legend style = { 
            legend pos = south west
          }
        ]
        \addplot table[x=meshsize,y={create col/linear regression={y=err_ch_en}}] {ch_0_pi6_tiltedhexagonal.dat}
        coordinate [pos=0.75] (A)
        coordinate [pos=1.00] (B);
        \xdef\slopea{\pgfplotstableregressiona}
        \draw (A) -| (B) node[pos=0.75,anchor=east] {\pgfmathprintnumber{\slopea}};
        \addplot table[x=meshsize,y={create col/linear regression={y=err_ch_en}}] {ch_1_pi6_tiltedhexagonal.dat}
        coordinate [pos=0.75] (A)
        coordinate [pos=1.00] (B);
        \xdef\slopeb{\pgfplotstableregressiona}
        \draw (A) -| (B) node[pos=0.75,anchor=south east] {\pgfmathprintnumber{\slopeb}};
        \addplot table[x=meshsize,y={create col/linear regression={y=err_ch_en}}] {ch_2_pi6_tiltedhexagonal.dat}
        coordinate [pos=0.75] (A)
        coordinate [pos=1.00] (B);
        \xdef\slopec{\pgfplotstableregressiona}
        \draw (A) -| (B) node[pos=0.75,anchor=east] {\pgfmathprintnumber{\slopec}};
        \legend{$k=0$,$k=1$,$k=2$};
      \end{loglogaxis}
    \end{tikzpicture}
    \subcaption{$\norm[a,h]{\uchn[N] - \Ih c^N}$ vs. $h$}
  \end{minipage}
  \\
  \vfill

  \begin{minipage}[b]{\textwidth}
    %% Triangular, err_wh_en
    \begin{tikzpicture}[scale=0.50]
      \begin{loglogaxis}[
          xmin = 0.001,
          legend style = {
            legend pos = south west
          }
        ]
        \addplot table[x=meshsize,y={create col/linear regression={y=err_wh_en}}] {ch_0_mesh1.dat}
        coordinate [pos=0.75] (A)
        coordinate [pos=1.00] (B);
        \xdef\slopea{\pgfplotstableregressiona}
        \draw (A) -| (B) node[pos=0.75,anchor=east] {\pgfmathprintnumber{\slopea}};
        \addplot table[x=meshsize,y={create col/linear regression={y=err_wh_en}}] {ch_1_mesh1.dat}
        coordinate [pos=0.75] (A)
        coordinate [pos=1.00] (B);
        \xdef\slopeb{\pgfplotstableregressiona}
        \draw (A) -| (B) node[pos=0.75,anchor=south east] {\pgfmathprintnumber{\slopeb}};
        \addplot table[x=meshsize,y={create col/linear regression={y=err_wh_en}}] {ch_2_mesh1.dat}
        coordinate [pos=0.75] (A)
        coordinate [pos=1.00] (B);
        \xdef\slopec{\pgfplotstableregressiona}
        \draw (A) -| (B) node[pos=0.75,anchor=east] {\pgfmathprintnumber{\slopec}};
        \legend{$k=0$,$k=1$,$k=2$};
      \end{loglogaxis}
    \end{tikzpicture}
    \hfill
    %% Cartesian, err_wh_en
    \begin{tikzpicture}[scale=0.50]
      \begin{loglogaxis}[
          xmin = 0.002,
          legend style = {
            legend pos = south west
          }
        ]
        \addplot table[x=meshsize,y={create col/linear regression={y=err_wh_en}}] {ch_0_mesh2.dat}
        coordinate [pos=0.75] (A)
        coordinate [pos=1.00] (B);
        \xdef\slopea{\pgfplotstableregressiona}
        \draw (A) -| (B) node[pos=0.75,anchor=east] {\pgfmathprintnumber{\slopea}};
        \addplot table[x=meshsize,y={create col/linear regression={y=err_wh_en}}] {ch_1_mesh2.dat}
        coordinate [pos=0.75] (A)
        coordinate [pos=1.00] (B);
        \xdef\slopeb{\pgfplotstableregressiona}
        \draw (A) -| (B) node[pos=0.75,anchor=south east] {\pgfmathprintnumber{\slopeb}};
        \addplot table[x=meshsize,y={create col/linear regression={y=err_wh_en}}] {ch_2_mesh2.dat}
        coordinate [pos=0.75] (A)
        coordinate [pos=1.00] (B);
        \xdef\slopec{\pgfplotstableregressiona}
        \draw (A) -| (B) node[pos=0.75,anchor=east] {\pgfmathprintnumber{\slopec}};
        \legend{$k=0$,$k=1$,$k=2$};
      \end{loglogaxis}
    \end{tikzpicture}
    \hfill
    %% Hexagonal, err_wh_en
    \begin{tikzpicture}[scale=0.50]
      \begin{loglogaxis}[
          xmin = 0.0025,
          legend style = {
            legend pos = south west
          }
        ]
        \addplot table[x=meshsize,y={create col/linear regression={y=err_wh_en}}] {ch_0_pi6_tiltedhexagonal.dat}
        coordinate [pos=0.75] (A)
        coordinate [pos=1.00] (B);
        \xdef\slopea{\pgfplotstableregressiona}
        \draw (A) -| (B) node[pos=0.75,anchor=east] {\pgfmathprintnumber{\slopea}};
        \addplot table[x=meshsize,y={create col/linear regression={y=err_wh_en}}] {ch_1_pi6_tiltedhexagonal.dat}
        coordinate [pos=0.75] (A)
        coordinate [pos=1.00] (B);
        \xdef\slopeb{\pgfplotstableregressiona}
        \draw (A) -| (B) node[pos=0.75,anchor=south east] {\pgfmathprintnumber{\slopeb}};
        \addplot table[x=meshsize,y={create col/linear regression={y=err_wh_en}}] {ch_2_pi6_tiltedhexagonal.dat}
        coordinate [pos=0.75] (A)
        coordinate [pos=1.00] (B);
        \xdef\slopec{\pgfplotstableregressiona}
        \draw (A) -| (B) node[pos=0.75,anchor=east] {\pgfmathprintnumber{\slopec}};
        \legend{$k=0$,$k=1$,$k=2$};
      \end{loglogaxis}
    \end{tikzpicture}
    \subcaption{$\norm[a,h]{\uwhn[N]-\Ih w^N}$ vs. $h$}
  \end{minipage}
  \caption{Energy-errors at final time vs. $h$. From left to right: triangular, Cartesian and (predominantly) hexagonal mesh families; cf. Figure~\ref{mesh_families}.\label{tab_cv:en}}
\end{figure}

%------------------------------------------------------------------------------%
% L2 errors
%------------------------------------------------------------------------------%

\begin{figure}
  \begin{minipage}[b]{\textwidth}
    %% Triangular, err_ch_L2
    \begin{tikzpicture}[scale=0.50]
      \begin{loglogaxis}[
          xmin = 0.001,
          legend style = {
            legend pos = south west 
          }
        ]
        \addplot table[x=meshsize,y={create col/linear regression={y=err_ch_L2}}] {ch_0_mesh1.dat}
        coordinate [pos=0.75] (A)
        coordinate [pos=1.00] (B);
        \xdef\slopea{\pgfplotstableregressiona}
        \draw (A) -| (B) node[pos=0.75,anchor=east] {\pgfmathprintnumber{\slopea}};
        \addplot table[x=meshsize,y={create col/linear regression={y=err_ch_L2}}] {ch_1_mesh1.dat}
        coordinate [pos=0.75] (A)
        coordinate [pos=1.00] (B);
        \xdef\slopeb{\pgfplotstableregressiona}
        \draw (A) -| (B) node[pos=0.75,anchor=south east,above=10pt] {\pgfmathprintnumber{\slopeb}};
        \addplot table[x=meshsize,y={create col/linear regression={y=err_ch_L2}}] {ch_2_mesh1.dat}
        coordinate [pos=0.75] (A)
        coordinate [pos=1.00] (B);
        \xdef\slopec{\pgfplotstableregressiona}
        \draw (A) -| (B) node[pos=0.75,anchor=east] {\pgfmathprintnumber{\slopec}};
        \legend{$k=0$,$k=1$,$k=2$};
      \end{loglogaxis}
    \end{tikzpicture}
    \hfill
    %% Cartesian, err_ch_L2
    \begin{tikzpicture}[scale=0.50]
      \begin{loglogaxis}[
          xmin = 0.002,
          legend style = {
            legend pos = south west
          }
        ]
        \addplot table[x=meshsize,y={create col/linear regression={y=err_ch_L2}}] {ch_0_mesh2.dat}
        coordinate [pos=0.75] (A)
        coordinate [pos=1.00] (B);
        \xdef\slopea{\pgfplotstableregressiona}
        \draw (A) -| (B) node[pos=0.75,anchor=east] {\pgfmathprintnumber{\slopea}};
        \addplot table[x=meshsize,y={create col/linear regression={y=err_ch_L2}}] {ch_1_mesh2.dat}
        coordinate [pos=0.75] (A)
        coordinate [pos=1.00] (B);
        \xdef\slopeb{\pgfplotstableregressiona}
        \draw (A) -| (B) node[pos=0.75,anchor=south east,above=10pt] {\pgfmathprintnumber{\slopeb}};
        \addplot table[x=meshsize,y={create col/linear regression={y=err_ch_L2}}] {ch_2_mesh2.dat}
        coordinate [pos=0.75] (A)
        coordinate [pos=1.00] (B);
        \xdef\slopec{\pgfplotstableregressiona}
        \draw (A) -| (B) node[pos=0.75,anchor=east] {\pgfmathprintnumber{\slopec}};
        \legend{$k=0$,$k=1$,$k=2$};
      \end{loglogaxis}
    \end{tikzpicture}
    \hfill
    %% Hexagonal, err_ch_L2
    \begin{tikzpicture}[scale=0.50]
      \begin{loglogaxis}[
          xmin = 0.0025,
          legend style = { 
            legend pos = south west
          }
        ]
        \addplot table[x=meshsize,y={create col/linear regression={y=err_ch_L2}}] {ch_0_pi6_tiltedhexagonal.dat}
        coordinate [pos=0.75] (A)
        coordinate [pos=1.00] (B);
        \xdef\slopea{\pgfplotstableregressiona}
        \draw (A) -| (B) node[pos=0.75,anchor=east] {\pgfmathprintnumber{\slopea}};
        \addplot table[x=meshsize,y={create col/linear regression={y=err_ch_L2}}] {ch_1_pi6_tiltedhexagonal.dat}
        coordinate [pos=0.75] (A)
        coordinate [pos=1.00] (B);
        \xdef\slopeb{\pgfplotstableregressiona}
        \draw (A) -| (B) node[pos=0.75,anchor=south east,above=10pt] {\pgfmathprintnumber{\slopeb}};
        \addplot table[x=meshsize,y={create col/linear regression={y=err_ch_L2}}] {ch_2_pi6_tiltedhexagonal.dat}
        coordinate [pos=0.75] (A)
        coordinate [pos=1.00] (B);
        \xdef\slopec{\pgfplotstableregressiona}
        \draw (A) -| (B) node[pos=0.75,anchor=east] {\pgfmathprintnumber{\slopec}};
        \legend{$k=0$,$k=1$,$k=2$};
      \end{loglogaxis}
    \end{tikzpicture}
    \subcaption{$\norm{c_h^N - \lproj{k+1} c^N}$ vs. $h$}
  \end{minipage}
  \\
  \vfill

  \begin{minipage}[b]{\textwidth}
    %% Triangular, err_wh_L2
    \begin{tikzpicture}[scale=0.50]
      \begin{loglogaxis}[
          xmin = 0.001,
          legend style = {
            legend pos = south west
          }
        ]
        \addplot table[x=meshsize,y={create col/linear regression={y=err_wh_L2}}] {ch_0_mesh1.dat}
        coordinate [pos=0.75] (A)
        coordinate [pos=1.00] (B);
        \xdef\slopea{\pgfplotstableregressiona}
        \draw (A) -| (B) node[pos=0.75,anchor=east] {\pgfmathprintnumber{\slopea}};
        \addplot table[x=meshsize,y={create col/linear regression={y=err_wh_L2}}] {ch_1_mesh1.dat}
        coordinate [pos=0.75] (A)
        coordinate [pos=1.00] (B);
        \xdef\slopeb{\pgfplotstableregressiona}
        \draw (A) -| (B) node[pos=0.75,anchor=south east,above=10pt] {\pgfmathprintnumber{\slopeb}};
        \addplot table[x=meshsize,y={create col/linear regression={y=err_wh_L2}}] {ch_2_mesh1.dat}
        coordinate [pos=0.75] (A)
        coordinate [pos=1.00] (B);
        \xdef\slopec{\pgfplotstableregressiona}
        \draw (A) -| (B) node[pos=0.75,anchor=east] {\pgfmathprintnumber{\slopec}};
        \legend{$k=0$,$k=1$,$k=2$};
      \end{loglogaxis}
    \end{tikzpicture}
    \hfill
    %% Cartesian, err_wh_L2
    \begin{tikzpicture}[scale=0.50]
      \begin{loglogaxis}[
          xmin = 0.002,
          legend style = {
            legend pos = south west
          }
        ]
        \addplot table[x=meshsize,y={create col/linear regression={y=err_wh_L2}}] {ch_0_mesh2.dat}
        coordinate [pos=0.75] (A)
        coordinate [pos=1.00] (B);
        \xdef\slopea{\pgfplotstableregressiona}
        \draw (A) -| (B) node[pos=0.75,anchor=east] {\pgfmathprintnumber{\slopea}};
        \addplot table[x=meshsize,y={create col/linear regression={y=err_wh_L2}}] {ch_1_mesh2.dat}
        coordinate [pos=0.75] (A)
        coordinate [pos=1.00] (B);
        \xdef\slopeb{\pgfplotstableregressiona}
        \draw (A) -| (B) node[pos=0.75,anchor=south east,above=10pt] {\pgfmathprintnumber{\slopeb}};
        \addplot table[x=meshsize,y={create col/linear regression={y=err_wh_L2}}] {ch_2_mesh2.dat}
        coordinate [pos=0.75] (A)
        coordinate [pos=1.00] (B);
        \xdef\slopec{\pgfplotstableregressiona}
        \draw (A) -| (B) node[pos=0.75,anchor=east] {\pgfmathprintnumber{\slopec}};
        \legend{$k=0$,$k=1$,$k=2$};
      \end{loglogaxis}
    \end{tikzpicture}
    \hfill
    %% Hexagonal, err_wh_L2
    \begin{tikzpicture}[scale=0.50]
      \begin{loglogaxis}[
          xmin = 0.0025,
          legend style = {
            legend pos = south west
          }
        ]
        \addplot table[x=meshsize,y={create col/linear regression={y=err_wh_L2}}] {ch_0_pi6_tiltedhexagonal.dat}
        coordinate [pos=0.75] (A)
        coordinate [pos=1.00] (B);
        \xdef\slopea{\pgfplotstableregressiona}
        \draw (A) -| (B) node[pos=0.75,anchor=east] {\pgfmathprintnumber{\slopea}};
        \addplot table[x=meshsize,y={create col/linear regression={y=err_wh_L2}}] {ch_1_pi6_tiltedhexagonal.dat}
        coordinate [pos=0.75] (A)
        coordinate [pos=1.00] (B);
        \xdef\slopeb{\pgfplotstableregressiona}
        \draw (A) -| (B) node[pos=0.75,anchor=south east,above=10pt] {\pgfmathprintnumber{\slopeb}};
        \addplot table[x=meshsize,y={create col/linear regression={y=err_wh_L2}}] {ch_2_pi6_tiltedhexagonal.dat}
        coordinate [pos=0.75] (A)
        coordinate [pos=1.00] (B);
        \xdef\slopec{\pgfplotstableregressiona}
        \draw (A) -| (B) node[pos=0.75,anchor=east] {\pgfmathprintnumber{\slopec}};
        \legend{$k=0$,$k=1$,$k=2$};
      \end{loglogaxis}
    \end{tikzpicture}
    \subcaption{$\norm{w_h^N-\lproj{k+1} w^N}$ vs. $h$}
  \end{minipage}
  \caption{$L^2$-errors at final time vs. $h$. From left to right: triangular, Cartesian and (predominantly) hexagonal mesh families; cf. Figure~\ref{mesh_families}.\label{tab_cv:L2}}
\end{figure}

\subsection{Evolution of an elliptic and a cross-shaped interfaces\label{num_ellipse}}

The numerical examples of this section consist in tracking the evolution of initial data corresponding, respectively, to an elliptic and a cross-shaped interface between phases.
For the elliptic interface test case of Figure~\ref{DG_ellipse}, the initial datum is
$$
c_0(\vec{x}) =
\begin{cases}
  0.95  & \text{if $81\left(x_1-0.5\right)^2 + 9\left(x_2-0.5\right)^2 < 1$,} \\
  -0.95 &\text{otherwhise.}
\end{cases}
$$
For the cross-shaped interface test case of Figure~\ref{DG_cross}, we take
$$
c_0(\vec{x})
=\begin{cases}
0.95 & \begin{array}{ll}
  \text{if} & 5\left(
  |(x_2{-}0.5) -\frac{2}{5}(x_1{-}0.5)|
  + |\frac{2}{5}(x_1{-}0.5) + (x_2{-}0.5)|
  \right) < 1
  \\
  \text{or} & 5\left(
  |(x_1{-}0.5) - \frac{2}{5}(x_2{-}0.5)|
  + |\frac{2}{5}(x_2{-}0.5) + (x_1{-}0.5)|
  \right) < 1,
\end{array}
\\
-0.95 &\text{otherwhise.}
\end{cases}
$$

In both cases, the space domain is the unit square $\Omega=(0,1)^2$, and the interface parameter $\gamma$ is taken to be $\pgfmathprintnumber{1e-2}$.
%% we apply the HHO method and compare the numerical solution with that of~\cite{Kay.Styles.ea:09}.
We use a $64\times64$ uniform Cartesian mesh and $k=1$ with time step $\tau=\gamma^2/10$.

In the test case of Figure~\ref{DG_ellipse}, we observe evolution of the elliptic interface towards a circular interface and, as expected, mass is well preserved (+0.5\% with respect to the initial ellipse).
Similar considerations hold for the cross-shaped test case of Figure~\ref{DG_cross}, which has the additional difficulty of presenting sharp corners.

\begin{figure}
  \centering
  \includegraphics[height=3.5cm]{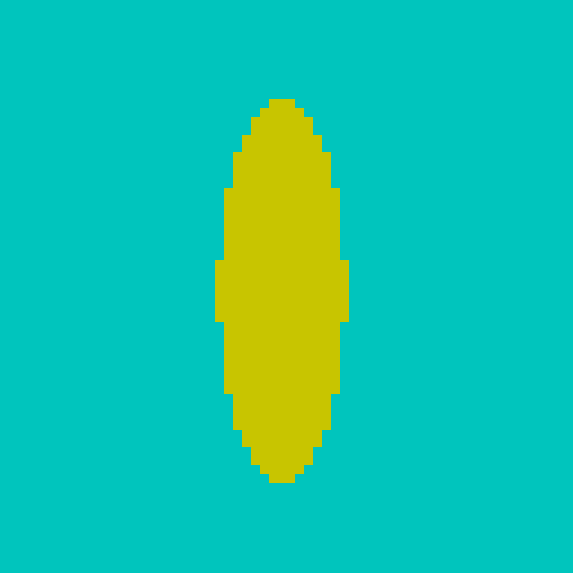}
  \hspace{0.25cm}
  \includegraphics[height=3.5cm]{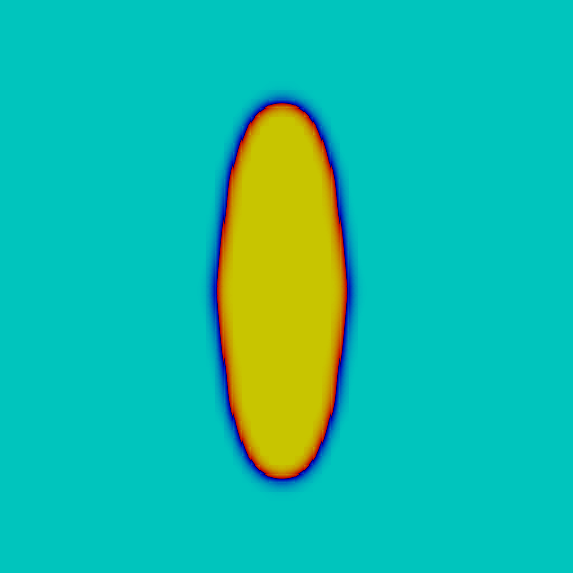}
  \\ \vspace{0.25cm}
  \includegraphics[height=3.5cm]{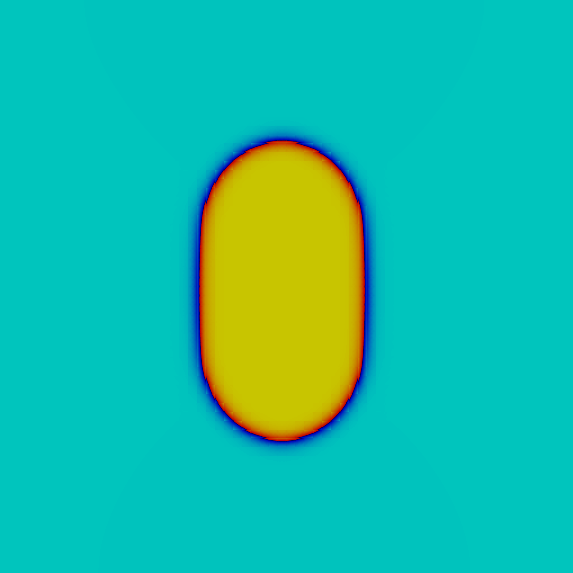}
  \hspace{0.25cm}
  \includegraphics[height=3.5cm]{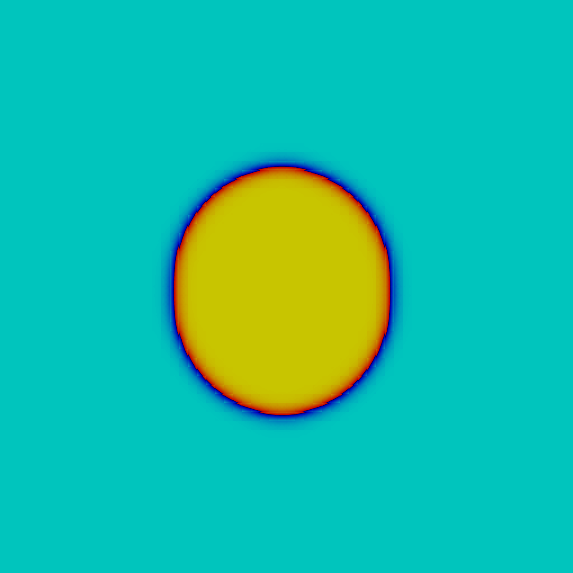}
  \caption{Evolution of an elliptic interface (left to right, top to bottom). Displayed times are 0 , $\pgfmathprintnumber{3e-3}$ , $\pgfmathprintnumber{0.3}$, 1.\label{DG_ellipse}}
\end{figure}

\begin{figure}
  \begin{center}
    \centering
    \includegraphics[height=3.5cm]{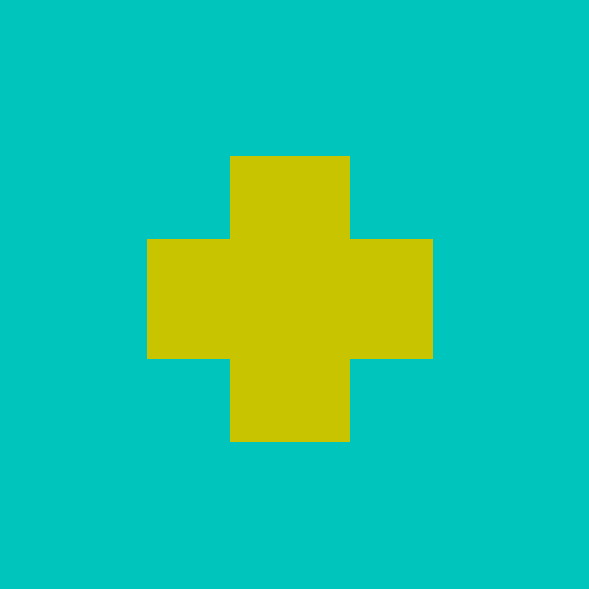}
    \hspace{0.25cm}
    \includegraphics[height=3.5cm]{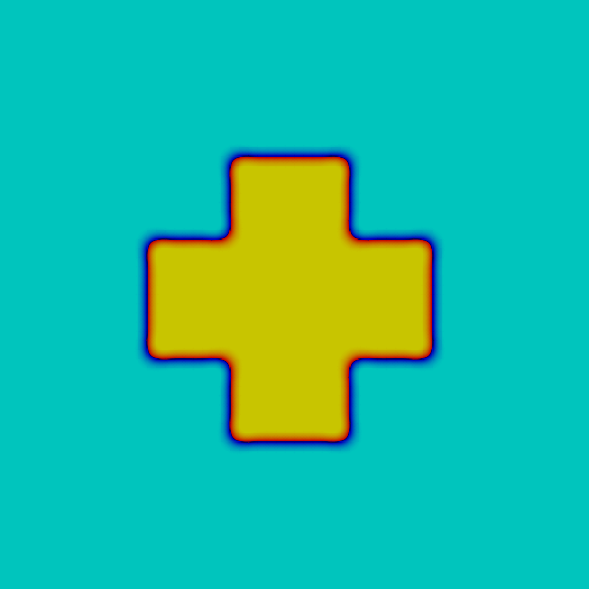}
    \\ \vspace{0.25cm}
    \includegraphics[height=3.5cm]{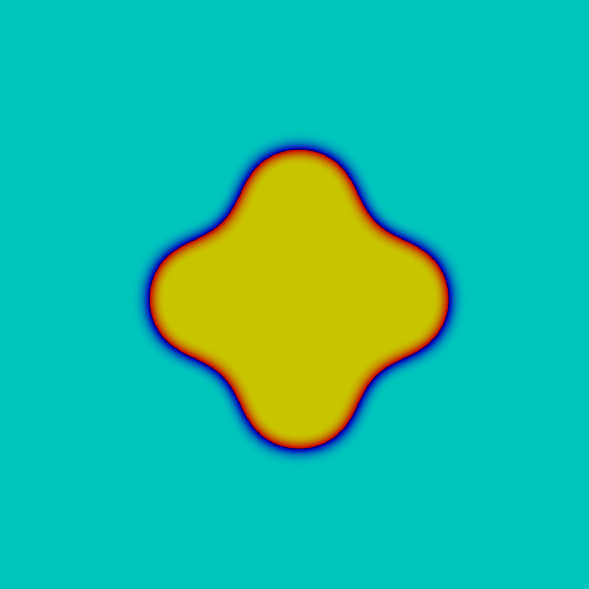}
    \hspace{0.25cm}
    \includegraphics[height=3.5cm]{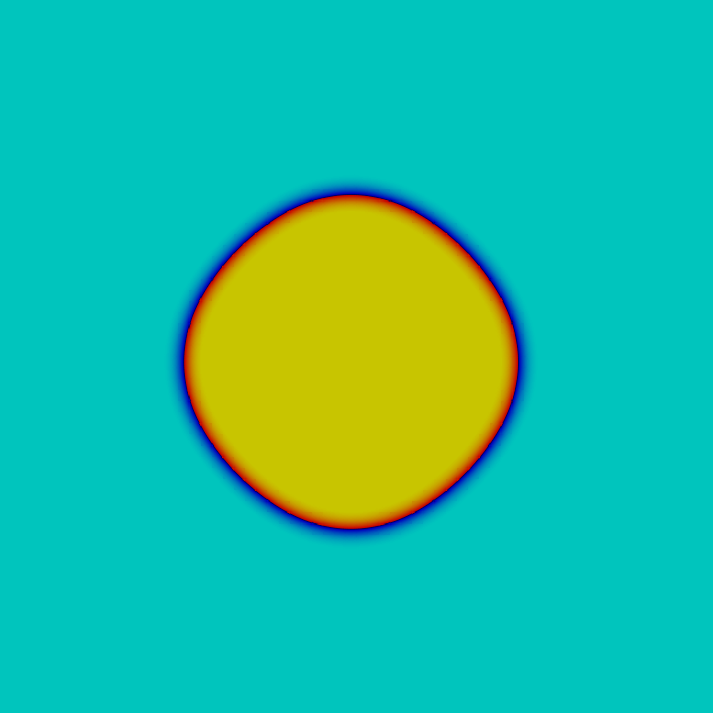}
  \end{center}
  \caption{Evolution of a cross-shaped interface (left to right, top to bottom). Displayed times are 0, $\pgfmathprintnumber{5e-5}$, $\pgfmathprintnumber{1e-2}$, $\pgfmathprintnumber{8.17e-2}$.\label{DG_cross}}
\end{figure}

\subsection{Spinodal decomposition\label{num_spinoid}}

Spinodal decomposition can be observed when a binary alloy is heated to a high temperature for a certain time and then abruptly cooled.
As a result, phases are separated in well-defined high concentration areas.
In Figure~\ref{fig:spinodal_decomp}, we display the numerical solutions obtained on a $128\times128$ uniform Cartesian mesh for $k=0$ and on a uniform $64\times64$ Cartesian mesh for $k=1$.
In both cases, we use the same initial conditions taking random values between -1 and 1 on a $32\times32$ uniform Cartesian partition of the domain.
The interface parameter is $\gamma =1/100$, and we take $\tau = \gamma^2/10$.
For $k=0$, the time discretisation is based on the Backward Euler scheme while, for $k=1$, we use the BDF2 formula to make sure that the spatial and temporal error contributions are equilibrated; cf. Remark~\ref{rem:bdf2}.

\begin{figure}
  \begin{minipage}[b]{\textwidth}\begin{center}
    \includegraphics[height=3.5cm]{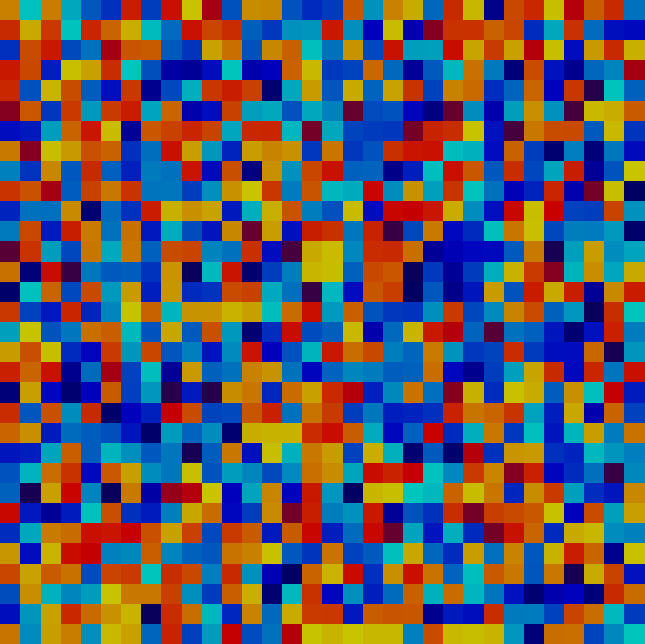}
    \hspace{0.25cm}
    \includegraphics[height=3.5cm]{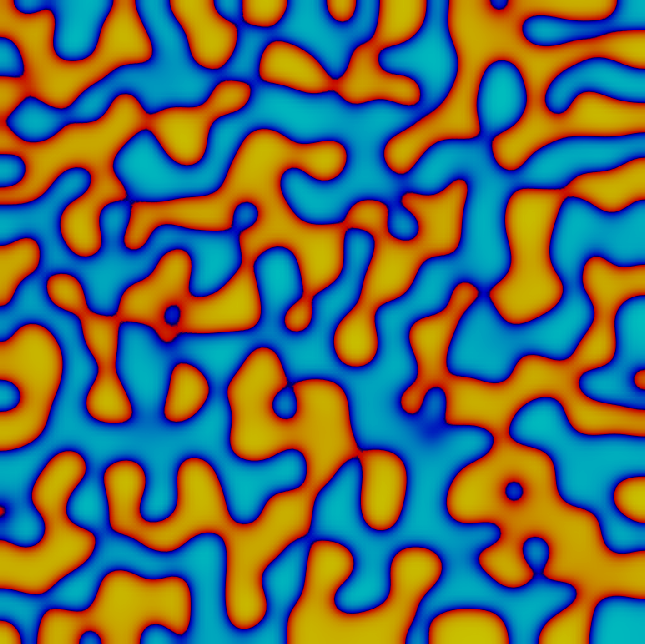}
    \\ \vspace{0.25cm}
    \includegraphics[height=3.5cm]{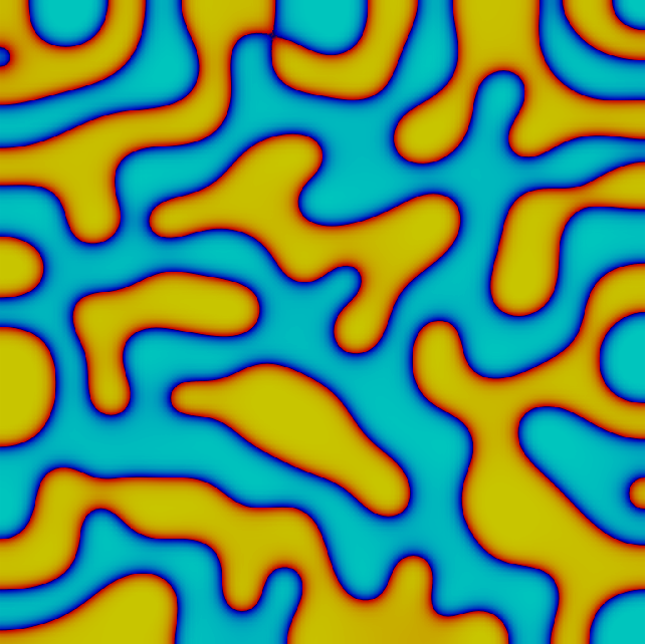}
    \hspace{0.25cm}
    \includegraphics[height=3.5cm]{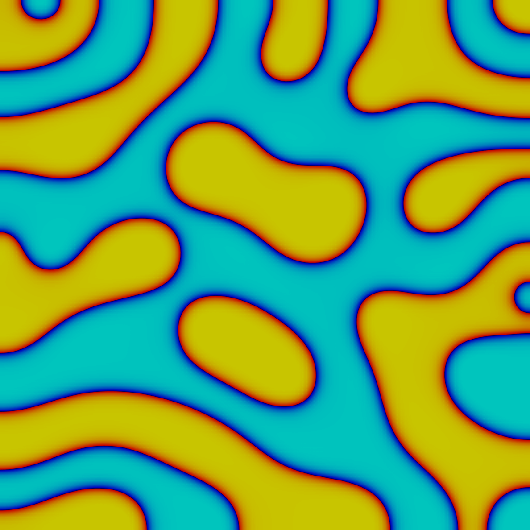}\end{center}
    \subcaption{$128\times128$ uniform Cartesian mesh, $k=0$, BE\label{fig:spinoidal_decomp:k=0}}
  \end{minipage}
  \vspace{0.125cm}\\
  \begin{minipage}[b]{\textwidth}\begin{center}
    \includegraphics[height=3.5cm]{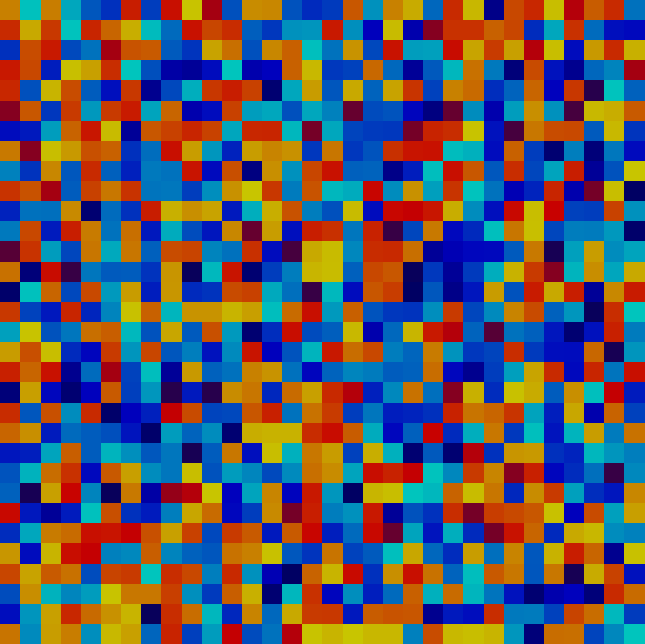}
    \hspace{0.25cm}
    \includegraphics[height=3.5cm]{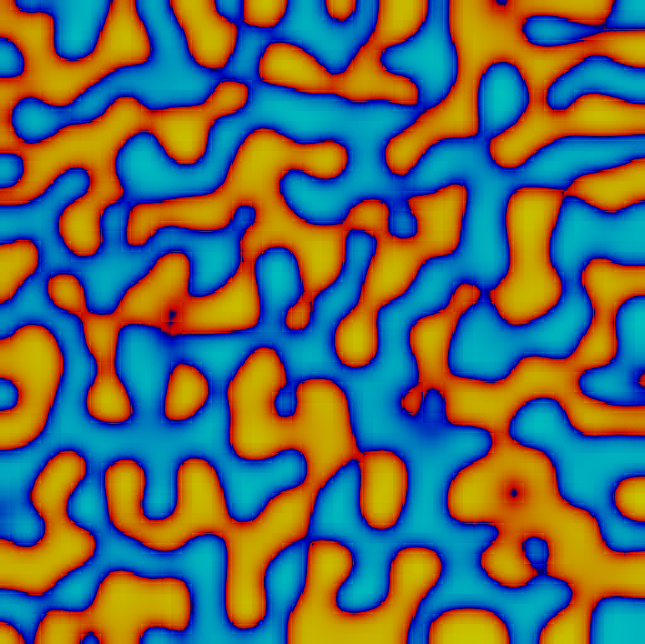}
    \\ \vspace{0.25cm}
    \includegraphics[height=3.5cm]{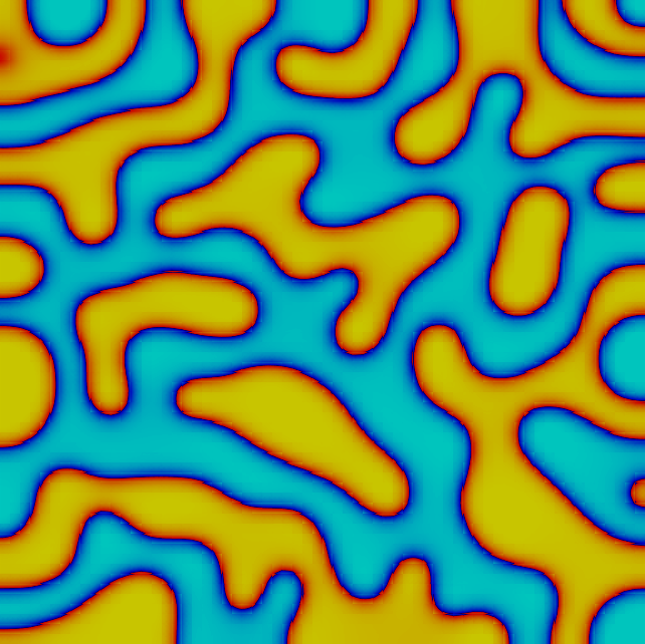}
    \hspace{0.25cm}
    \includegraphics[height=3.5cm]{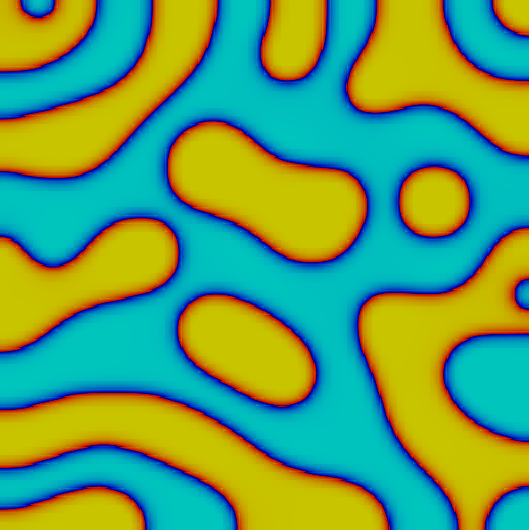}\end{center}
    \subcaption{$64\times64$ uniform Cartesian mesh, $k=1$, BDF2\label{fig:spinoidal_decomp:k=1}}
  \end{minipage}
  \caption{Spinoidal decomposition (left to right, top to bottom). In both cases, the same random initial condition is used. Displayed times are 0, $\pgfmathprintnumber{5e-5}$, $\pgfmathprintnumber{1.25e-3}$, $\pgfmathprintnumber{3.6e-2}$.\label{fig:spinodal_decomp}}
\end{figure}

The separation of the two components into two distinct phases happens over a very small time; see two leftmost panels of Figure \ref{fig:spinodal_decomp} corresponding to times $0$ and $\pgfmathprintnumber{5e-5}$, respectively.
Later, the phases gather increasingly slowly until the interfaces develop a constant curvature; see the two rightmost panels of Figure~\ref{fig:spinodal_decomp}, corresponding to times $\pgfmathprintnumber{1.25e-3}$ and $\pgfmathprintnumber{3.6e-2}$, respectively.
At the latest stages, we can observe that the solution exhibits a (small) dependence on the mesh and/or the polynomial degree, and the high-concentration regions in Figures~\ref{fig:spinoidal_decomp:k=0} and~\ref{fig:spinoidal_decomp:k=1} are highly superposable but not identical.

%------------------------------------------------------------------------------%
%
\appendix
\section{Proofs of discrete functional analysis results}\label{sec:proofs}

This section contains the proofs of Lemmas~\ref{lem:agmon} and~\ref{lem:gnp} preceeded by the required preliminary technical results.

\begin{proposition}[Estimates for $\uLh$]\label{prop:est.uLh}
  Assuming mesh quasi-uniformity~\eqref{eq:qu}, it holds
  \begin{alignat}{2}
    \label{eq:est.L2.uLh}
    \forall\uvh&\in\Uh,\qquad&\norm[0,h]{\uLh\uvh}\lesssim h^{-1}\norm[1,h]{\uvh},
    \\
    \label{eq:est.H-1.uLh}
    \forall \uvh&\in\UhO,\qquad&\norm[H^{-1}(\Omega)]{\Lh\uvh}\lesssim\norm[1,h]{\uvh}.
  \end{alignat}
\end{proposition}

\begin{proof}
  \begin{asparaenum}[(i)]
  \item \emph{Proof of~\eqref{eq:est.L2.uLh}.}
    Let $\uvh\in\Uh$.
    Making $\uzh=-\uLh\uvh$ in the definition~\eqref{eq:uLh} of $\uLh$, we have
    $$
    \norm[0,h]{\uLh\uvh}^2
    = -a_h(\uvh,\uLh\uvh)
    \lesssim\norm[1,h]{\uvh}\norm[1,h]{\uLh\uvh}
    \lesssim\norm[1,h]{\uvh}h^{-1}\norm[0,h]{\uLh\uvh},
    $$
    where we have used the continuity of $a_h$ expressed by the second inequality in~\eqref{eq:norm1h.ah} followed by the fact that, for all $\uzh\in\Uh$, $\norm[1,h]{\uzh}\lesssim h^{-1}\norm[0,h]{\uzh}$.
    This inequality follows from the definition~\eqref{eq:norm1h} of the $\norm[1,h]{{\cdot}}$-norm using the inverse inequality~\eqref{eq:inv} to bound the first term and recalling mesh quasi-uniformity~\eqref{eq:qu}.
    
  \item \emph{Proof of~\eqref{eq:est.H-1.uLh}.}
    Let $\uvh\in\UhO$.
    Observing that $\Lh\uvh$ has zero-average on $\Omega$ (cf. Remark~\ref{rem:restr.uLh}), we have
    \begin{equation}\label{eq:est.uLh:1}
      \norm[H^{-1}(\Omega)]{\Lh\uvh}
      = \sup_{\varphi\in H^1(\Omega)\cap L^2_0(\Omega),\norm[H^1(\Omega)]{\varphi}=1}(\Lh\uvh,\varphi).
    \end{equation}
    Let now $\uphi\eqbydef\Ih\varphi$.
    Using the fact that $\Lh\uvh\in\Poly{k+1}(\Th)$ followed by the definitions~\eqref{eq:uLh} of $\uLh$ and~\eqref{eq:prod0h} of $(\cdot,\cdot)_{0,h}$, one has
    $$
    (\Lh\uvh,\varphi)
    =(\Lh\uvh,\lproj{k+1}\varphi)
    = -s_{0,h}(\uLh\uvh,\uphi) - a_h(\uvh,\uphi).
    $$
    Hence, using the Cauchy--Schwarz inequality we get
    $$
    \begin{aligned}
      |(\Lh\uvh,\varphi)|
      &\lesssim\seminorm[0,h]{\Lh\uvh}\seminorm[0,h]{\uphi} + \norm[1,h]{\uvh}\norm[1,h]{\uphi}
      \\
      &\lesssim h^{-1}\norm[1,h]{\uvh} h\seminorm[1,h]{\uphi} + \norm[1,h]{\uvh}\norm[1,h]{\uphi}
      \\
      &\lesssim \norm[1,h]{\uvh}\norm[1,h]{\uphi}
      \lesssim \norm[1,h]{\uvh}\norm[H^1(\Omega)]{\varphi},
    \end{aligned}
    $$
    where we have used the second inequality in~\eqref{eq:norm1h.ah} in the first line,~\eqref{eq:est.L2.uLh} together with the fact that $\seminorm[0,h]{\uzh}\le h\seminorm[1,h]{\uzh}$ for all $\uzh\in\Uh$ to pass to the second line, and the $H^1$-stability~\eqref{eq:Ih.stab} of $\Ih$ to conclude.
    To obtain~\eqref{eq:est.H-1.uLh}, plug the above estimate into the right-hand side of~\eqref{eq:est.uLh:1}.
  \end{asparaenum}
\end{proof}

We introduce the continuous Green's function $\G:L_0^2(\Omega)\to H^1(\Omega)\cap L^2_0(\Omega)$ such that, for all $\varphi\in L_0^2(\Omega)$,
$$
(\GRAD\G \varphi,\GRAD v)=(\varphi,v)\qquad\forall v\in H^1(\Omega).
$$
Owing to elliptic regularity (which holds since $\Omega$ is convex), we have $\G\varphi\in H^2(\Omega)$.
Its discrete counterpart $\uGh:\UhO\to\UhO$ is defined such that, for all $\uphi\in\UhO$,
\begin{equation}
  \label{eq:uGh}
  a_h(\uGh\uphi,\uzh) = (\uphi,\uzh)_{0,h}\qquad\forall\uzh\in\UhO,
\end{equation}
with inner product $(\cdot,\cdot)_{0,h}$ defined by~\eqref{eq:prod0h}.
We will denote by $\Gh\uvh$ (no underline) the broken polynomial function in $\Poly{k+1}(\Th)$ obtained from element DOFs in $\uGh\uvh$.
We next show that $-\uGh$ is the inverse of $\uLh$ restricted to $\UhO\to\UhO$.
Let $\uvh\in\UhO$.
Using~\eqref{eq:uGh} with $\uphi=\uLh\uvh$ followed by~\eqref{eq:uLh}, it is inferred, for all $\uzh\in\UhO$,
$$
a_h(\uGh\uLh\uvh,\uzh)
=(\uLh\uvh,\uzh)_{0,h}
=-a_h(\uvh,\uzh)\implies a_h(\uvh+\uGh\uLh\uvh,\uzh)=0.
$$
Therefore, since  $(\uvh+\uGh\uLh\uvh)\in\UhO$ and $a_h$ is coercive in $\UhO$ (cf.~\eqref{eq:norm1h.ah} and Proposition~\ref{prop:norm1h}), we conclude
\begin{equation}
  \label{eq:zh.uGh}
  \uvh + \uGh\uLh\uvh=\underline{0}\qquad\forall\uvh\in\UhO.
\end{equation}

\begin{proposition}[Estimates for $\uGh$]\label{prop:stab.approx.uGh}
  It holds, for all $\uvh\in\UhO$,
  \begin{equation}
    \label{eq:stab.green}
    \norm[1,h]{\uGh\uvh-\Ih\G v_h}
    \lesssim h\left(\seminorm[0,h]{\uvh} + \norm[H^2(\Omega)]{\G v_h}\right)
    \lesssim h\norm[0,h]{\uvh}.
  \end{equation}
  Moreover, using elliptic regularity, we have
  \begin{equation}
    \label{eq:approx.green}
    \norm{\Gh\uvh-\lproj{k+1}\G v_h}
    \lesssim h^2\left(\seminorm[0,h]{\uvh} + \norm[H^2(\Omega)]{\G v_h}\right)
    \lesssim h^2\norm[0,h]{\uvh}.
  \end{equation}
\end{proposition}

\begin{proof}
  Let $\uvh\in\UhO$.
  \begin{asparaenum}[(i)]
  \item \emph{Proof of~\eqref{eq:stab.green}.}
    For all $\uzh\in\UhO$ we have, using the definition~\eqref{eq:uGh} of $\uGh\uvh$ and subtracting the quantity $(v_h+\LAPL\G v_h,z_h)=0$,
    \begin{equation}\label{eq:stab.uGh:1}
      a_h(\uGh\uvh-\Ih\G v_h,\uzh)
      = \underbrace{(\uvh,\uzh)_{0,h} - (v_h,z_h)}_{\term_1}
      \underbrace{-a_h(\Ih\G v_h,\uzh) - (\LAPL\G v_h,z_h)}_{\term_2}.
    \end{equation}
    Recalling the definition~\eqref{eq:prod0h} of the inner product $(\cdot,\cdot)_{0,h}$, one has
    \begin{equation}\label{eq:stab.uGh:T1}
      |\term_1|
      =|s_{0,h}(\uvh,\uzh)|
      \le\seminorm[0,h]{\uvh}\seminorm[0,h]{\uzh}
      \le h\seminorm[0,h]{\uvh}\seminorm[1,h]{\uzh}.
    \end{equation}
    On the other hand, the consistency property~\eqref{eq:cons.ah} of the bilinear form $a_h$ readily yields
    \begin{equation}\label{eq:stab.uGh:T2}
      |\term_2|
      \lesssim h\norm[H^2(\Omega)]{\G v_h}\norm[1,h]{\uzh}.
    \end{equation}
    Making $\uzh=\uGh\uvh-\Ih\G v_h$ in~\eqref{eq:stab.uGh:1}, and using the coercivity of $a_h$ expressed by the first inequality in~\eqref{eq:norm1h.ah} followed by the bounds~\eqref{eq:stab.uGh:T1}--\eqref{eq:stab.uGh:T2}, the first bound in~\eqref{eq:stab.green} follows.
    To prove the second bound in~\eqref{eq:stab.green}, use elliptic regularity to estimate $\norm[H^2(\Omega)]{\G v_h}\lesssim\norm{v_h}$ and recall the definition of the $\norm[0,h]{{\cdot}}$-norm.

  \item \emph{Proof of~\eqref{eq:approx.green}.}
    We follow the ideas of~\cite[Theorem~10]{Di-Pietro.Ern.ea:14} and~\cite[Theorem~11]{Di-Pietro.Ern:15}, to which we refer for further details.
    Set, for the sake of brevity, $\uphi\eqbydef\uGh\uvh-\Ih\G v_h$, and let $z\eqbydef\G\varphi_h$.
    By elliptic regularity, $z\in H^2(\Omega)$ and $\norm[H^2(\Omega)]{z}\lesssim\norm{\varphi_h}$.
    Observing that $-\LAPL z=\varphi_h$, letting $\uhzh\eqbydef\Ih z$, and using the definition~\eqref{eq:uGh} of $\uGh$, we have
    \begin{equation}\label{eq:approx.green:basic}
      \norm{\varphi_h}^2
      = \underbrace{-(\LAPL z,\varphi_h)-a_h(\uphi,\uhzh)}_{\term_1}
      + \underbrace{(v_h,\hzh) - a_h(\Ih\G v_h,\uhzh)}_{\term_2}
      + \underbrace{s_{0,h}(\uvh,\uhzh)}_{\term_3}.
    \end{equation}
    Using the consistency~\eqref{eq:cons.ah} of $a_h$, it is readily inferred for the first term 
     \begin{equation}\label{eq:approx.green:T1}
       |\term_1|\lesssim h\norm[H^2(\Omega)]{z}\norm[1,h]{\uphi}\lesssim
       h^2\left(\seminorm[0,h]{\uvh} + \norm[H^2(\Omega)]{\G v_h}\right)\norm{\varphi_h},
     \end{equation}
     where we have used elliptic regularity to infer $\norm[H^2(\Omega)]{z}\lesssim\norm{\varphi_h}$ and~\eqref{eq:stab.green} to bound $\norm[1,h]{\uphi}$.
     For the second term, upon observing that $(v_h,\hzh)=-(\LAPL\G v_h,z)=(\GRAD\G v_h,\GRAD z)$ since, by definition of, $-\LAPL\G v_h=v_h\in\Poly{k+1}(\Th)$ and $\hzh=\lproj{k+1} z$, recalling the definition~\eqref{eq:ah} of the bilinear form $a_h$ and using the orthogonality property~\eqref{eq:pT.IT.ell.proj} of $(\pT\circ\IT)$, we have
     $$
     \term_2
     = 
     \sum_{T\in\Th} (\GRAD(\pT\IT\G v_h - \G v_h), \GRAD(\pT\uhzh - z))_T
     + s_{1,h}(\Ih\G v_h,\uhzh).
     $$
     By the approximation properties of $(\pT\circ\IT)$ and of $\lproj{k+1}$, and bounding $\norm[H^2(\Omega)]{z}$ and $\norm[1,h]{\uphi}$ as before, we have
     \begin{equation}\label{eq:approx.green:T2}
       |\term_2| \lesssim h^2\left(\seminorm[0,h]{\uvh} + \norm[H^2(\Omega)]{\G v_h}\right)\norm{\varphi_h}.
     \end{equation}
    Finally, for the last term, we write
    \begin{equation}\label{eq:approx.green:T3}
      |\term_3|\le\seminorm[0,h]{\uvh}\seminorm[0,h]{\uhzh}
      \lesssim \seminorm[0,h]{\uvh} h^2\norm[H^2(\Omega)]{z}
      \lesssim h^2\seminorm[0,h]{\uvh}\norm{\varphi_h},
    \end{equation}
    where we have used the Cauchy--Schwarz inequality in the first bound, the approximation properties~\eqref{eq:lproj.approx} of $\lproj{k+1}$ in the second bound, and elliptic regularity to conclude.
    Using~\eqref{eq:approx.green:T1}--\eqref{eq:approx.green:T3} to estimate the right-hand side of~\eqref{eq:approx.green:basic} the first inequality in \eqref{eq:approx.green} follows. Using elliptic regularity to further bound $\norm[H^2(\Omega)]{\G v_h}\lesssim\norm{v_h}$ and recalling the definition of the $\norm[0,h]{{\cdot}}$-norm yields the second inequality in \eqref{eq:approx.green}.
  \end{asparaenum}
\end{proof}

\begin{remark}[Choice of $s_{0,h}$]\label{rem:s0h}
  The choice~\eqref{eq:prod0h} for the stabilisation bilinear form $s_{0,h}$ is crucial to have the right-hand side of~\eqref{eq:approx.green:T3} scaling as $h^2$.
  Penalizing the full difference $(v_F-v_T)$ instead of the lowest-order part $\lproj[F]{k}(v_F-v_T)$ would have lead to a right-hand side only scaling as $h$.
\end{remark}

We are now ready to prove Lemma~\ref{lem:agmon}.

\begin{proof}[Proof of Lemma~\ref{lem:agmon}]
  Let $\uvh\in\UhO$ and set $\uphi\eqbydef\uLh\uvh$.
  Recalling that, owing to~\eqref{eq:zh.uGh}, $v_h=-\Gh\uphi$, it is inferred using the triangle inequality,
  \begin{equation}
    \label{eq:agmon:1}  
    \norm[L^\infty(\Omega)]{v_h}
    \le\norm[L^\infty(\Omega)]{\lproj{k+1}\G\varphi_h} 
    + \norm[L^\infty(\Omega)]{\Gh\uphi-\lproj{k+1}\G\varphi_h}
    \eqbydef\term_1 + \term_2.
  \end{equation}
  The $L^\infty$-stability of $\lproj{k+1}$ (cf.~\eqref{eq:lproj.stab}) followed by the continuous Agmon's inequality readily yields for the first term
  \begin{equation}\label{eq:agmon.T1}
    \term_1\lesssim
    \norm[L^\infty(\Omega)]{\G\varphi_h}
    \lesssim\norm[H^1(\Omega)]{\G\varphi_h}^{\frac12}\norm[H^2(\Omega)]{\G\varphi_h}^{\frac12}.
  \end{equation}
  Using a standard regularity shift (cf., e.g.,~\cite{Grisvard:92}), recalling that $\varphi_h=\Lh\uvh$, and using the $H^{-1}$-bound~\eqref{eq:est.H-1.uLh} for $\Lh\uvh$, we have
  \begin{equation}\label{eq:agmon.est.G}
    \norm[H^1(\Omega)]{\G\varphi_h}\lesssim\norm[H^{-1}(\Omega)]{\varphi_h}\lesssim\norm[1,h]{\uvh},\qquad
    \norm[H^2(\Omega)]{\G\varphi_h}\lesssim\norm{\varphi_h}=\norm{\Lh\uvh},
  \end{equation}
  which plugged into~\eqref{eq:agmon.T1} yields
  \begin{equation}\label{eq:agmon.T1'}
    \term_1\lesssim\norm[1,h]{\uvh}^{\frac12}\norm{\Lh\uvh}^{\frac12}.
  \end{equation}
  For the second term we have, on the other hand,
  \begin{equation}\label{eq:agmon.T2}
    \begin{aligned}  
      \term_2
      &\lesssim h^{-\frac{d}{2}}\norm{\Gh\uphi-\lproj{k+1}\G\varphi_h}
      \\
      &\lesssim h^{\frac{3-d}{2}}(h\norm[0,h]{\uLh\uvh})^{\frac12}\norm[0,h]{\uLh\uvh}^{\frac12}
      \\
      &\lesssim h^{\frac{3-d}{2}}\norm[1,h]{\uvh}^{\frac12}\norm[0,h]{\uLh\uvh}^{\frac12}
      \lesssim \norm[1,h]{\uvh}^{\frac12}\norm[0,h]{\uLh\uvh}^{\frac12},
   \end{aligned}
  \end{equation}
  where we have used the global inverse inequality~\eqref{eq:glob.inv.Linfty} with $p=2$ to obtain the first bound, the estimate \eqref{eq:approx.green} to obtain the second, \eqref{eq:est.L2.uLh} to obtain the third, and the fact that $d\le 3$ together with $h\le h_\Omega\lesssim 1$ (with $h_\Omega$ diameter of $\Omega$) to conclude.
  The conclusion follows plugging~\eqref{eq:agmon.T1'} and~\eqref{eq:agmon.T2} into~\eqref{eq:agmon:1}.
\end{proof}

\begin{remark}[Discrete Agmon's inequality in dimension $d=2$]
  When $d=2$, we have the following sharper form for the discrete Agmon's inequality:
  \begin{equation}\label{eq:agmon:d=2}
    \forall\uvh\in\UhO,\qquad      
    \norm[L^\infty(\Omega)]{v_h}\lesssim\norm[0,h]{\uvh}^{\frac12}\norm[0,h]{\uLh\uvh}^{\frac12}.
  \end{equation}
  To obtain~\eqref{eq:agmon:d=2}, the following modifications are required in the above proof:
  \begin{inparaenum}[(i)]
  \item The term $\term_1$ is bounded as
    $
    \term_1\lesssim\norm{\G\varphi_h}^{\frac12}\norm[H^2(\Omega)]{\G\varphi_h}
    \lesssim\norm{v_h}^{\frac12}\norm{\Lh\uvh}^{\frac12},
    $
    where we have used $v_h=-\G\varphi_h$ (cf.~\eqref{eq:zh.uGh}) for the first factor and~\eqref{eq:agmon.est.G} for the second;
  \item The third line of~\eqref{eq:agmon.T2} becomes $\term_2\lesssim (h\norm[1,h]{\uvh})^{\frac12}\norm[0,h]{\uLh\uvh}^{\frac12}\lesssim\norm[0,h]{\uvh}^{\frac12}\norm[0,h]{\uLh\uvh}^{\frac12}$, where we have used the inverse inequality~\eqref{eq:inv} and mesh quasi-uniformity to bound the first factor.
  \end{inparaenum}
\end{remark}

%------------------------------------------------------------------------------%

We next prove the discrete Gagliardo--Nirenberg--Poincar\'{e}'s inequality of Lemma~\ref{lem:gnp}.

\begin{proof}[Proof of Lemma~\ref{lem:gnp}]
  Using the same notation as in the proof of Lemma~\ref{lem:agmon}, we have
  \begin{equation*}\label{eq:gnp:1}
    \norm[L^p(\Omega)^d]{\GRADh v_h}
    \le
    \norm[L^p(\Omega)^d]{\GRADh\lproj{k+1}\G\varphi_h}
    + \norm[L^p(\Omega)^d]{\GRADh(\Gh\uphi-\lproj{k+1}\G\varphi_h)}
    \eqbydef\term_1+\term_2.
  \end{equation*}
  For the first term, we use the $W^{1,p}$-stability of $\lproj{k+1}$ (cf.~\eqref{eq:lproj.stab}) followed by the continuous Gagliardo--Nirenberg--Poincar\'{e}'s inequality~\eqref{eq:gnp.cont}, and~\eqref{eq:agmon.est.G} to infer
  $$
  \term_1
  \lesssim\seminorm[W^{1,p}(\Omega)]{\G\varphi_h}
  \lesssim\seminorm[H^1(\Omega)]{\G\varphi_h}^{1-\alpha}\norm[H^2(\Omega)]{\G\varphi_h}^{\alpha}
  \lesssim\norm[1,h]{\uvh}^{1-\alpha}\norm{\Lh\uvh}^{\alpha}.
  $$ 
  For the second term, on the other hand, we have
  $$
  \begin{aligned}
    \term_2
    &\lesssim h^{d\left(\frac1p-\frac12\right)}\norm{\GRADh(\Gh\uphi-\lproj{k+1}\G\varphi_h)}
    \\
    &\lesssim h^{d\left(\frac1p-\frac12\right)}\norm[1,h]{\uGh\uphi-\Ih\G \varphi_h}^{1-\alpha} \norm[1,h]{\uGh\uphi-\Ih\G \varphi_h}^{\alpha}
    \\
    &\lesssim h^{\alpha+d\left(\frac1p-\frac12\right)}(
      h\norm[0,h]{\uLh\uvh}
    )^{1-\alpha}\norm[0,h]{\uLh\uvh}^{\alpha}
    \\
    &\lesssim h^{\alpha+d\left(\frac1p-\frac12\right)}\norm[1,h]{\uvh}^{1-\alpha}\norm[0,h]{\uLh\uvh}^{\alpha}
    \lesssim \norm[1,h]{\uvh}^{1-\alpha}\norm[0,h]{\uLh\uvh}^{\alpha},
  \end{aligned}
  $$
  where we have used the global reverse Lebesgue inequality~\eqref{eq:glob.inv} in the first line, the definition \eqref{eq:norm1h} of the $\norm[1,h]{{\cdot}}$-norm to pass to the second line, the estimate \eqref{eq:stab.green} to pass to the third line, and~\eqref{eq:est.L2.uLh} to pass to the fourth line.
  To obtain the second inequality in the fourth line, we observe that, recalling the definition~\eqref{eq:gnp.cont} of $\alpha$ and the assumptions on $p$, it holds for the exponent of $h$,
    $$
    \alpha + d\left(\frac1p-\frac12\right)
    = \frac12 - \frac{d}2\left(\frac12-\frac1p\right)\ge 0,
    $$
  and, since $h\le h_\Omega\lesssim 1$, the conclusion follows.
\end{proof}

\begin{remark}[Validity of the discrete Agmon's and Gagliardo--Niremberg--Poincar\'e's inequalities]
  At the discrete level, the fact that the discrete Agmon's inequality~\eqref{eq:agmon} is valid only up to $d=3$ and that the Gagliardo--Nirenberg--Poincar\'e's inequalities~\eqref{eq:gnp} are valid only for $p\in[2,+\infty)$ if $d=2$, $p\in[2,6]$ if $d=3$ is reflected by the need to have nonnegative powers of $h$ in the estimates of the terms $\term_2$ to conclude in the corresponding proofs.
\end{remark}

%------------------------------------------------------------------------------%

\bibliographystyle{siamplain}
\bibliography{chho}

\end{document}